\documentclass[12pt,letterpaper,titlepage,reqno]{amsart}
\usepackage{amsmath, amssymb, amsthm, amsfonts,amscd,xr}
\usepackage{mathtools}
\usepackage{doc}
\usepackage{hyperref}
\hypersetup{colorlinks=true,citecolor=blue,linkcolor=blue,filecolor=blue,urlcolor=blue}

\def\beq{\begin{equation}}
\def\eeq{\end{equation}}

\def\e{\varepsilon}

\newenvironment{res}
               {\begin{equation}
\begin{minipage}{0.85\textwidth}}
               { \end{minipage}\end{equation} }
\def\ber{\begin{res} }
\def\eer{\end{res}}

\numberwithin{equation}{section}

\newtheorem{thm}{Theorem}[section]
\newtheorem{theorem}[thm]{Theorem}

\newtheorem{cor}[thm]{Corollary}

\newtheorem{proposition}[thm]{Proposition}

\theoremstyle{definition}

\newtheorem{remark}[thm]{Remark}
\newtheorem{df}[thm]{Definition}

\makeatletter
\def\section{\@startsection {section}{1}{\z@}{3.5ex plus 1ex minus
    .2ex}{2.3ex plus .2ex}{\large\bf}}
    \def\subsection{\@startsection{subsection}{2}{\z@}{3.25ex plus 1ex minus
 .2ex}{1.5ex plus .2ex}{\bf}}
\makeatother

\newcommand{\fa}{\mathfrak{a}}
\newcommand{\fb}{\mathfrak{b}}

\newcommand{\fg}{\mathfrak{g}}
\newcommand{\fh}{\mathfrak{h}}

\newcommand{\fk}{\mathfrak{k}}

\newcommand{\fn}{\mathfrak{n}}

\newcommand{\fp}{\mathfrak{p}}

\newcommand{\ft}{\mathfrak{t}}

\newcommand{\ffR}{\mathfrak{R}}

\newcommand{\fU}{\mathfrak{U}}

\newcommand{\fZ}{\mathfrak{Z}}
\newcommand{\bba}{\mathbb{A}}

\newcommand{\bbc}{\mathbb{C}}
\newcommand{\bbd}{\mathbb{D}}

\newcommand{\bbh}{\mathbb{H}}

\newcommand{\bbn}{\mathbb{N}}

\newcommand{\bbq}{\mathbb{Q}}
\newcommand{\bbr}{\mathbb{R}}

\newcommand{\bbz}{\mathbb{Z}}
\newcommand{\bfz}{\mathfrak{Z}}
\newcommand{\Ca}{\mathcal{A}}
\newcommand{\Cb}{\mathcal{B}}
\newcommand{\Cc}{\mathcal{C}}
\newcommand{\Cd}{\mathcal{D}}
\newcommand{\Ce}{\mathcal{E}}
\newcommand{\Cf}{\mathcal{F}}

\newcommand{\Cj}{\mathcal{J}}
\newcommand{\Ck}{\mathcal{K}}

\newcommand{\Cs}{\mathcal{S}}

\newcommand{\Cz}{\mathcal{Z}}

\newcommand{\bA}{\mathbf{A}}

\newcommand{\bH}{\mathbf{H}}

\newcommand{\ga}{\alpha}
\newcommand{\gb}{\beta}
\newcommand{\gc}{\gamma}
\newcommand{\gd}{\delta}

\newcommand{\gt}{\tau}
\newcommand{\gl}{\lambda}

\newcommand{\go}{\omega}

\newcommand{\gep}{\epsilon}

\newcommand{\cpv}{\begin{center}\begin{minipage}[t]{14cm}\small{\it Proof.} }
\newcommand{\fpv}{\fim\end{minipage}\end{center}}
\newcommand{\fim}{\hfill $\Box$}

\DeclareMathOperator{\Hom}{Hom}

\DeclareMathOperator{\End}{End}

\newcommand {\sfk}{\mathsf{k}}

\newcommand {\Omit}[1]{}
\newcommand {\Omitcomment}[1]{}

\newcommand {\nocomment}[1]{}

\newcommand {\wrt}{with respect to }
\newcommand {\wt}[1]{\widetilde{#1}}

\newcommand {\nak}{^{\operatorname{\scriptscriptstyle{NAK}}}}

\newcommand {\aand}{\qquad\text{and}\qquad}

\newcommand {\ie}{{\it i.e., }}

\begin{document}
\title[On Harish--Chandra's Isomorphism] 
{On Harish--Chandra's Isomorphism}
\dedicatory{On the occasion of Harish--Chandra's centenary celebration}
\author[E. Opdam]{Eric Opdam}
\address{Korteweg de Vries Institute for Mathematics\\
University of Amsterdam\\
Science Park 904, 1098 XH Amsterdam, The Netherlands\\
email: e.m.opdam@uva.nl}
\author[V. Toledano Laredo]{Valerio Toledano Laredo}
\thanks{VTL was supported in part through the NSF grant DMS--2302568.}
\address{Department of Mathematics, Northeastern University\\ 
360 Huntington Avenue, Boston, MA, 02115, USA\\
email: v.toledanolaredo@neu.edu}

\Omit{ArXiv abstract:
This is the text of a talk given by the first author at the Harish-Chandra centenary meeting held in Allahabad
in October 2023. It reviews Harish-Chandra's isomorphism and its many applications to representation theory
and mathematical physics. It also announces the existence and uniqueness of nonsymmetric shift operators
for an arbitrary root system. These are differential--reflection operators with a transmutation property relative to
Dunkl--Cherednik operators: they shift the parameter k of these operators by 1, and restrict on symmetric
functions to the hypergeometric shift operators introduced by the first author.
}

\begin{abstract}
Harish--Chandra's famous isomorphism (1957) gives a description of the algebra of invariant differential 
operators of a Riemannian globally symmetric space $X$ as a polynomial algebra with $n$ indeterminates, where 
$n$ denotes the rank of $X$. This key result has inspired many developments connecting
representation theory, mathematical physics, special functions and algebraic geometry. After discussing the classical 
context for real and $p$--adic groups, we focus on a generalisation of Harish--Chandra's system of differential equations 
for zonal spherical functions, called the hypergeometric system of differential equations for root systems. An  
important role is played by graded affine Hecke algebras, which creates various interesting new perspectives. We close with 
a discussion of the existence and uniqueness of {\it nonsymmetric} shift operators, which are certain differential--reflection 
operators satisfying a ``transmutation property'' relative to Dunkl--Cherednik operators.
\end{abstract}
\keywords{Harish--Chandra homomorphism, Satake isomorphism, Affine Hecke algebra, Dunkl--Cherednik operator, Shift operator}
\subjclass[2000]{Primary 20C08; Secondary 22D25, 22E35, 43A30}
\maketitle\tableofcontents
\section{Introduction}
Let $G_c$ be a connected complex reductive group, and $\Cz(G_c)$ the algebra of linear differential
operators on $G_c$ which are invariant under $G_c \times G_c$ acting by left and right translations on $G_c$.   
In his 1951 paper \cite{HC1}, Harish--Chandra gave an explicit description of $\Cz(G_c)$, a result which is 
fundamental in representation theory. To describe the result, let $\fg_c$ denote the Lie algebra of $G_c$. 
We identify the universal enveloping algebra $\fU$ of $\fg_c$ with the algebra of left-invariant differential operators on $G_c$. 
Under this identification $\Cz(G_c)$ corresponds to the algebra $\fU^{\textup{Ad}(G_c)}$ of ${\textup{Ad}(G_c)}$--invariant 
elements of $\fU$, which equals in our setting to the center 
$\fZ=\fU^{\textup{Ad}(G_c)}\subset \fU$ of $\fU$. 
Let $\fh_c\subset \fg_c$ be a Cartan subalgebra of $\fg_c$, and let $S(\fh_c)\xhookrightarrow{}\fU$ be the canonical 
embedding of the symmetric algebra $S(\fh_c)=\fU(\fh_c)$ of $\fh_c$ as a subalgebra of $\fU$. 
Harish--Chandra \cite[Part III]{HC1} showed that there is an algebra isomorphism 
\begin{equation}
\gc: \Cz(G_c)\xrightarrow{\simeq} S(\fh_c)^W
\end{equation}
where $W$ denotes the Weyl group of $(\fg_c,\fh_c)$. 


The interest of Harish--Chandra in the homomorphism $\gamma$ is revealed in the introduction of
his 1957 paper \cite{HC2}, where he sketched with broad strokes an inspired vision for his long--term
goal of computing the explicit Plancherel formula for a real reductive group $G$. A key role is played
by the study  of orbital integrals, and their relation to the homomorphism $\gamma$. 

To wit, let $\theta$ denote a Cartan involution of $G$, and $H$ a $\theta$--stable Cartan subgroup
of $G$ with Lie algebra $\fh\subset\fg$. We denote the root system of $(\fg_\bbc,\fh_\bbc)$ by $\Phi
\subset \fh_\bbc^*$, and fix a set of positive roots $\Phi_+\subset \Phi$. 
Harish--Chandra introduced orbital integrals relative to a Cartan subgroup $H$ of $G$ 
in \cite{HC2}. We shall assume for simplicity that the Weyl vector $\delta=\frac{1}{2}\sum_{\ga\in\Phi_+}\ga$ is analytically 
integral on $H$. The Weyl denominator $\Delta(h)=\prod_{\ga\in \Phi_+}(h^{\ga/2}-h^{-\ga/2})$ is then a well--defined real analytic function 
on $H$.

Let $H'=\{h\in H \mid \Delta(h)\not=0\}\subset H$ be the open and dense subset of regular elements.
Harish--Chandra defined the orbital integral of a function $f\in C_c^\infty(G)$ as the function $F_f$ on $H'$
given by
\begin{equation}
F_f(h)=\gep_+(h)\Delta(h)\int_{G/H}f(x^*h(x^*)^{-1})dx^* 
\end{equation}
where $\gep_+(h)\in \{\pm 1\}$ is the sign of the product $\Delta_+(h)$ of the factors of $\Delta(h)$ corresponding 
to the non-imaginary roots. He then moved on to show that $F_f\in C^\infty(H')$ has bounded support, is a bounded function,  
and extends to a smooth function on the larger set $H''\subset H$ where only the factors of $\Delta$ corresponding 
to the noncompact imaginary roots are nonvanishing.

Subsequently, he showed that for $f\in C^\infty_c(G)$ and $z\in \Cz(G)$, the following remarkable formula holds
on $H'$
\begin{equation}
F_{zf}(h)=\gamma(z)F_f(h)
\end{equation}
Since $\gamma(z)$ is an invariant differential operator for the abelian group $H$, we see that the transform $f\to
F_f$ has a "transmutation property": it connects the harmonic analysis on $G$ in the open subset of  $G$ of elements
which are conjugate to elements of $H'$, with the harmonic analysis on the abelian group $H$. Harish--Chandra
deduced from this that in situations where $F_f$ extends to a compactly supported function on $H$, 
its Fourier transform at a character of the torus $\chi\in\hat{H}$ is a $G$--invariant eigendistribution. 
Of course, one needs to adopt this construction to the general situation, where $F_f$ does have jumps when 
moving from one connected component of $H''$ to a neighbouring one. This is a serious challenge, but fortunately 
the case $G=\textup{SL}_2(\mathbb{R})$ provides a guideline of how one could go about this.  

The picture that emerges in \cite{HC2} is quite compelling. Harish--Chandra realised, already at that very early
stage, the crucial importance of compact Cartan subgroups for the existence and parameterisation of  the discrete
series. In fact, the discrete series characters of $G$ should be related to Fourier transforms of orbital integrals 
relative to a compact Cartan subgroup. At the other extreme, the most continuous part of the Plancherel formula
of $G$ is parameterised by the Fourier transforms of the orbital integrals relative to a maximally split $\theta$--stable
Cartan subgroup.

A case in point is the $K$--spherical principal series, where $K\subset G$ is a maximal compact subgroup, which
constitutes a component of the most continuous part of the Plancherel decomposition. This component is studied
via the Fourier transform for $K$--spherical functions, a subject with very classical origins. 
This subject had already been studied by Bargmann \cite{Bar}, Gelfand \cite{Gel}, Gelfand--Naimark \cite{GeNa}, Godement \cite{Go}  
and indeed Harish--Chandra himself \cite{HC4}, as a natural generalisation of the theory of spherical harmonics. 
On the one hand, the compactness of $K$ makes it easy to see that the spectral measure of the $C^*$--completion 
of the commutative convolution algebra $L^1(G/\!\!/K)$\footnote{We denote by $G/\!\!/K$ the double coset space
$K\backslash G/K$.} acting on $L^2(G/\!\!/K)$ is a component of the most continuous part 
of the Plancherel measure of $G$. On the other hand, and this 
is one of the important insights Harish--Chandra provides in \cite{HC3}, this is the Plancherel measure of 
the Riemannian symmetric space $X=G/K$, which can be expressed 
in terms of the asymptotic behaviour of the elementary spherical functions "at infinity". He thus obtained a description 
of the spherical Plancherel measure in terms of the geometric and algebraic structure of $X$.  
The orbital integrals in this spherical case are smooth, thus the complication of the "jumps" of $F_f$ disappears. 
These ideas formed the basis of Harish--Chandra's approach to the determination of 
the explicit Plancherel formula for real reductive groups. 
\subsection*{Acknowledgments}
It is a pleasure to thank Dan Ciubotaru for his help on the topic of spherical unitary representations 
and interesting discussions.
\section{Harish--Chandra's homomorphism}\label{sec:HCH}
Harish--Chandra's homomorphism $\gc$ has a natural generalisation in the context of symmetric spaces, viewing 
$G$ as the symmetric space $G\times G/\textup{diag}(G)$. Of special interest to us is the case 
of a Riemannian symmetric space $X=G/K$, a case which Harish--Chandra worked out in \cite{HC3} in order to study the 
$K$--spherical part of the Plancherel measure of $G$. This is the main case we will be discussing in this paper.
There are further generalisations to much more general $G_c$-varieties, see \cite{Knop}.

To fix notations, let $G$ be a linear non-compact real reductive group, with Cartan involution $\theta$, and Cartan decomposition $G=KP$. 
Here $K=G^\theta$ is a maximal compact subgroup of $G$, and $P=\textup{exp}(\fp)$ where $\fp\subset \fg$ is the 
$-1$-eigenspace of $\theta$.\footnote{We will denote $d\theta\in\textup{GL}(\fg)$ simply by $\theta$} Recall that $\exp:\fp\to P$ is a 
diffeomorphism. For the Riemannian globally symmetric space $X=G/K$, we may identify the algebra 
$\bbd(X)$ of $G$--invariant differential operators on $X$ with 
$\fU^{\textup{Ad}(K)}/(\fU^{\textup{Ad}(K)}\cap \fU\fk)$. 
Let $\fh\subset \fg$ be a maximally split $\theta$-stable Cartan subalgebra of $\fg$, and let $\fa=\fh\cap \fp$ be the corresponding 
maximal abelian subspace of $\fp$. In this section, let $W\subset \textup{GL}(\fa)$ denote denote the little Weyl group, 
i.e. 
\begin{equation}\label{eq:W}
W=N_K(\fa)/C_K(\fa)
\end{equation} 
The set $\Phi$ of nonzero restrictions to $\fa$ of the roots in $\Phi(\fg_\bbc,\fh_\bbc)$ form a 
(possibly non--reduced) root system whose associated reflection group is equal to $W$. 
The Cartan decomposition $G=KP$ can be refined to the decomposition $G=KAK$, using that $P=\bigcup_{k\in K}kAk^{-1}$. 
If $g\in G$ and we write $g=k_1ak_2$, then the $W$-orbit $Wa\subset A$ is uniquely determined by $g$. 

\subsection{Zonal spherical functions}

A {\it spherical function} on $G$ is a function which is bi-$K$--invariant. By the Cartan decomposition $G=KAK$, such a
function satisfies $f(\theta(g))=f(g^{-1})$. Since $\theta$ is an automorphism of $G$  and $g\to g^{-1}$ is an anti-automorphism
of $G$, this observation yield a first fundamental result on spherical functions due to Gelfand \cite{Gel}. 

\begin{theorem}(\cite{Gel})
The convolution algebra $C_c(G/\!\!/K)$ of spherical continuous compactly supported functions is a commutative algebra. 
\end{theorem}
Let us recall some elements of the theory of spherical functions (cf. \cite{Go}). A joint eigenfunction $\varphi$ \wrt
convolution by elements of $C_c(G/\!\!/K)$ defines a homomorphism of $\bbc$-algebras $\Cf_{sph}^\varphi:C_c(G/\!\!/K)\to
\bbc$ such that for any $f\in C_c(G/\!\!/K)$ we have $f*\varphi=\Cf_{sph}^\varphi(f)\,\varphi$. 
It is not hard to see that the converse also holds: for any $\bbc$-algebra homomorphism $\Cf: C_c(G/\!\!/K)\to \bbc$, there
is a joint eigenfunction $\varphi\in C(G/\!\!/K)$ such that $\Cf=\Cf_{sph}^\varphi$. Observe that necessarily $\varphi(1)\not=0$. 

We call a joint eigenfunction $\varphi$ which is normalised by $\varphi(1)=1$ an {\it elementary} or {\it zonal}
spherical function. In that case, one easily sees that the corresponding joint eigenvalue $\Cf^\varphi_{sph}$ is
given by
\begin{equation}\label{eq:transf}
\Cf^\varphi_{sph}(f)=\int_G f(g)\varphi(g^{-1})dg
\end{equation}
In fact $\varphi$ is a zonal spherical function if and only if $\varphi$ satisfies the functional equation: 
\begin{equation}\label{eq:zonal}
\int_{K}\varphi(xky)dk=\varphi(x)\varphi(y)
\end{equation}
Indeed, for any $f\in C_c(G/\!\!/K)$ and $\varphi$ zonal spherical we have that: 
\begin{equation*} 
f*\varphi(y)
=\int_G f(g)\varphi(g^{-1}y)dg
=\int_G f(g)\int_K \varphi(g^{-1}ky)dkdg
\end{equation*}
and also 
\begin{equation*}
f*\varphi(y)
=\Cf^\varphi_{sph}(f)\varphi(y)
=\int_G f(g)\varphi(g^{-1})\varphi(y)dg
\end{equation*}
For any $y\in G$ fixed, this holds for all $f\in C_c(G/\!\!/K)$ if and only if (\ref{eq:zonal}) holds (with $x=g^{-1}$), 
since both hand sides are spherical functions in $x$, finishing the proof.

In a similar vein as the commutativity proof of $C_c(G/\!\!/K)$, it is not hard to see that for an invariant differential operator
$D\in \bbd(X)$, the formal transpose $D^t$ is equal to ${}^\theta{D}\in \bbd(X)$. 
Hence $\bbd(X)$ is commutative as well. Harish--Chandra's homomorphism gives an explicit description of this commutative algebra. 
First we fix a choice of positive roots of $\Phi(\fg_\bbc,\fh_\bbc)_+$, which also 
determines a set $\Phi_+\subset \Phi$ of positive roots restricted roots. We write $N\subset G$ for corresponding unipotent group, 
which gives rise to the Iwasawa decomposition $G=KAN$, where $A=\textup{exp}(\fa)$. 
We are now ready to formulate Harish--Chandra's theorem: 
\begin{thm}\label{thm:HChom}(\cite[Thm. 1]{HC3}) 
Given $D\in \bbd(X)$, define 
\[\gc(D):={}^\gt(\gc'(D))\in S(\fa_\bbc)\]
where $\gc'(D)$ 
is the component of 
$D\in \bbd(X)=\fU^{\textup{Ad}(K)}/(\fU^{\textup{Ad}(K)}\cap \fU\fk)$ in $S(\fa_\bbc)$ relative
to the decomposition $\fU=S(\fa_\bbc)\oplus \fk\fU\oplus \fU\fn$, and $\gt$ denotes the translation
over $\rho$, where $\rho\in\fa^*$ is  the half sum of the restricted positive roots $\Phi_+$
(i.e. $\gc(D)(\gl):=\gc'(D)(\gl-\rho)$).

Then, $\gamma$ defines an algebra isomorphism 
\begin{equation}
\gc:\bbd(X)\xrightarrow{\simeq} S(\fa_\bbc)^W
\end{equation}
\end{thm}
We remark that in this case of a maximally 
split Cartan $\fh$, the imaginary roots of $\Phi(\fg_\bbc,\fh_\bbc)$ are all compact. The restriction of this set of roots 
to $\fb:=\fh\cap\fk$ is a root system, namely the set of roots of the compact group $M:=C_K(\fa)$. The minimal $\theta$-stable Levi 
subgroup $C_G(\fa)=MA$ is the Levi factor of the minimal $\theta$-stable parabolic $Q=MAN\subset G$.  

\subsection{An integral formula for zonal spherical functions}\label{ss:D spherical}
The elementary spherical functions can also be characterised as analytic $K$--invariant functions $\varphi$ on $X$ such that $\varphi(1)=1$ and 
which are joint eigenfunctions of $\bbd(X)$. It thus follows that: 

\begin{cor}\label{co:spherical D}
The set of elementary spherical functions is 
parameterised by the set $W\backslash \fa_\bbc^*$ of $W$-orbits of elements $\gl\in \fa_\bbc^*$. The
elementary functions $\varphi_\gl$ which corresponds to $W\gl$ is the unique analytic $K$--invariant 
unction on $X$ satisfying:    
\begin{itemize}
\item[(i)] $D\varphi_\gl=\gc(D)(\gl)\varphi_\gl$ for any $D\in \bbd(X)$.
\item[(ii)] $\varphi_\gl(1)=1$. 
\end{itemize}
\end{cor}
From this description and the definition of $\gamma$, it is not hard to see that we have the following integral equation for $\varphi_\gl$:
\begin{proposition} 
Let $H:G\to A$ be the projection relative to the Iwasawa decomposition $G=KAN$. Then, for any $\gl\in\fa_\bbc^*$,
\begin{equation}\label{eq:phiHC}
\varphi_\gl(g)=\int_{K}e^{(\gl-\rho)(\textup{log}(H(gk)))}dk
\end{equation}
where $dk$ is the normalised Haar measure of $K$.
\end{proposition}
One shows easily that (see \cite[Ch IV.4, Lemma 4.4, formula (7)]{Helgason}) that $\varphi_\gl(g^{-1})=\varphi_{-\gl}(g)$. This leads to 
the following alternative integral formula for $\varphi_\gl$. By abuse of notation, we also use $A$ to denote the projection onto the $A$-component relative 
to the decomposition $G=NAK$. Then:
\begin{equation}\label{eq:phiHel} 
\varphi_\gl(g)=\int_{K}e^{(\gl+\rho)(\textup{log}(A(kg)))}dk
\end{equation} 
This also gives an alternative description of $\gc$, by computing the eigenvalue of $\varphi_\gl$ in two different ways using (\ref{eq:phiHC}) and (\ref{eq:phiHel}):
\begin{cor}\label{cor:nak} Let $D\in \bbd(X)$. 
\begin{itemize}
\item[(i)] Define
\[\gc\nak(D):=\gt^-(\gc''(D))\in S(\fa_\bbc)\]
where $\gc''(D)\in S(\fa_\bbc)$ is the component of 
$D$ in $S(\fa_\bbc)$ relative to the decomposition 
$\fU=S(\fa_\bbc)\oplus \fn\fU\oplus \fU\fk$, and $\gt^-$ denotes the translation over $-\rho$. Then, $\gc\nak(D)=\gc(D)$
for all $D\in \bbd(X)$.
\item[(ii)] We have $\gc(D^t)=(\gc(D))^t$, where $t$ refers to the formal transpose relative to $dg$ on the left hand side, and 
relative to $da$ on the right hand side.
\end{itemize}
\end{cor}
\subsection{The Abel transform}
The Abel transform of a function $f\in C_c^\infty(X)^K$ is a slight variation of the orbital integral of
$f$ relative to $H$ in this context. The difference is in the domain (the Abel transform is only defined
on the split component $A$ of $H$), and the normalisation factor (factors of roots which vanish on 
$\fa$ are omitted). Harish--Chandra defined \cite{HC3} for $a\in A$: 
\begin{equation}
F_f(a)=|\Delta_+(a)|\int_{G/H}f(x^*a(x^*)^{-1})dx^* 
\end{equation}
He shows that this can be rewritten as 
\begin{equation}\label{eq:abel}
F_f(a)=a^{\rho}\int_{N}f(an)dn
\end{equation}
where $dn$ is the Haar measure of $N$ such that $\int_{n\in N}e^{-2\rho{\textup{log}(H(\theta(n))}}dn=1$.

The Abel transform is a powerful tool, also in this context of the spherical Plancherel formula. At this point 
it is easy to show that: 
\begin{proposition}\label{prop:ImAb}
Let $f\in C_c^\infty(G/\!\!/K)$. We have:
\begin{itemize}
\item[(i)] (\cite{HC3}) $F_f\in C_c^\infty(A)^W$, and the map
\[F_f:C^\infty_c(G/\!\!/K)\to C_c^\infty(A)^W\]
is continuous. 
\item[(ii)] (\cite[Ch II.5(39)]{Helgason})\footnote{To prove this result, use Corollary \ref{cor:nak}(i).} 
For all $D\in \bbd(X)$ one has ($a\in A$): 
\begin{equation}
F_{Df}(a)=\gamma(D)F_f(a).
\end{equation}
\end{itemize}
\end{proposition}
We again see that $F_f$ connects the harmonic analysis on $X$ and the classical harmonic analysis 
on the abelian group $A$. The following result makes this precise. Let $dg$ denote the Haar measure on $G$ 
defined by 
\begin{equation}
\int_{G}f(g)dg=\int_{N}\int_{A}\int_{K}f(kan)e^{-2\rho(\textup{log}(a))}dkdadn
\end{equation}
If $f\in C_c^\infty(G/\!\!/K)$ and $\gl\in\fa_\bbc^*$ then we define its spherical Fourier transform $\Cf(f)^G_{sph}(\gl)$ by:
\begin{equation}
\Cf^G_{sph}(f)(\gl):=\Cf^{\varphi_\gl}_{sph}(f)=\int_G f(g)\varphi_\gl(g^{-1})dg
\end{equation}
If $g\in C_c^\infty(A)^W$ then we define the abelian Fourier transform of $g$ by:
\begin{equation}
\Cf^A(g)(\gl):=\int_A g(a)e^{-\gl(\textup{log}(a))}da
\end{equation}
The following result is straightforward from the definitions. 
\begin{proposition}\label{prop:transmut}(\cite{HC3})
For $f\in C_c^\infty(G/\!\!/K)$ and $\gl\in \fa_\bbc^*$ we have:
\begin{equation} 
\Cf^G_{sph}(f)(\gl):=\int_{G}\varphi_\gl(g^{-1})f(g)dg=\int_{A}e^{-\gl(\textup{log}(a))}F_f(a)da=\Cf^A(F_f)(\gl)
\end{equation}
\end{proposition}
\begin{cor}
We have $\varphi_\gl=\varphi_{w(\gl)}$ for all $w\in W$. 
\end{cor}
\subsection{The radial part of an invariant operator}
In order to study the spherical Fourier transform, Harish--Chandra studied the asymptotic behaviour
of $\varphi_{i\gl}$ on $G$ in \cite{HC3}. 
The Cartan decomposition $G=KAK$ and (\ref{eq:W}) makes it clear that a spherical function $\varphi$ 
on $G$ is determined by its restriction $\textup{Res}(\varphi_{i\gl})$ to $A$, and that its restriction is $W$--invariant. 
Therefore it is enough to study the asymptotic behaviour of $\varphi_{i\lambda}$ on $\overline{A_+}$, where 
\begin{equation}\label{eq:Aplus}
A_+:=\{a\in A\mid |\ga(a)|>1\textrm{\ for\ all\ positive\ roots\ }\ga\} 
\end{equation}
This is studied by a careful analysis of the differential equations which $\textup{Res}(\varphi_\gl)$ satisfy on $A$. 

Let $o=eK\in X$ denote the base point of $X=G/K$. 
Recall the $K$-{\it radial part} $\textup{Rad}(D)$ of an element $D\in \bbd(X)$, the radial part of $D$ 
when expressed in "polar coordinates" on $X$ relative to the action of $K$ on $X$, with $A.o\subset X$ 
as transversal set.
Then $\textup{Rad}(D)$ is a $W$--invariant linear partial differential operator 
on $A$ which has rational coefficients in the coordinates $x_i=e^{-\ga_i}$, with an asymptotic expansion of the form 
\begin{equation}\label{eq:hot}
\textup{Rad}(D)=\partial(\gc'(D))+\sum_{\mu\in Q:\mu<0}e^{\mu}\partial(p^D_\mu)
\end{equation}
where $\gc'(D)$ is defined in Theorem \ref{thm:HChom}, and $p^D_\mu\in S(\fa_\bbc)$ has degree 
less than $\gc'(D)$. 
The radial part map 
\begin{equation}
\textup{Rad}:\bbd(X)\to \textup{Rad}(\bbd(X))
\end{equation}
is an injective homomorphism of algebras. As an important explicit example let $L_X,L_A$ be the Laplace
operators of $X,A$ respectively, and recall that: 
\begin{equation}\label{eq:RadLX}
\textup{Rad}(L_X)=L_A+\sum_{\ga\in R_+}k_\ga \textup{coth}(\ga/2)\partial_{\ga^*}
\end{equation}
where $ R=2\Phi$, and where for each $\ga\in  R$, we define $k_\ga=\frac{1}{2}m_{\ga/2}$. 
(Here $\ga^*\in\fa$ is such that for all $\gb\in \fa^*$ we have $\gb(\ga^*)=(\ga,\gb)$; recall that 
$\gc(L_X)(\gl)=(\gl,\gl)-(\rho,\rho)$.)
Then the restriction $\textup{Res}({\varphi_\gl})$ of $\varphi_\gl$ to $A$ satisfies the system of equations (with $D\in \bbd(X)$): 
\begin{equation}\label{eq:spheq}
\textup{Rad}(D)(\textup{Res}({\varphi_\gl}))=\gc(D)(\gl)\textup{Res}({\varphi_\gl})
\end{equation}
\subsection{Asymptotic behaviour; Harish--Chandra's $c$--function}
Harish--Chandra deduced from the surjectivity of his homomorphism $\gamma$ and (\ref{eq:hot}) 
that the dimension of the local solution spaces of (\ref{eq:spheq}) on $A'$ is bounded by $|W|$. 
For $\gl$ generic, using the explicit form (\ref{eq:RadLX}) of $\textup{Rad}(L_X)$, he constructed a 
basis of solutions $\{\Phi_{w\gl}\}_{w\in W}$ of the system (\ref{eq:spheq}) on $A_+.o$ \cite[Section 8]{HC3}, where $\Phi_{w\gl}$ is 
a series of the form (with $a\in A_+$): 
\begin{equation}
\Phi_{w\gl}(a.o) =a^{w\gl-\rho}\sum_{\mu\in Q_-}\Gamma_\mu(\gl)a^{\mu}
\end{equation}
with $\Gamma_0(\gl)=1$. The series $\Phi_{w\gl}$ is easily shown to be uniformly convergent on all subsets of the form 
$\overline{A_+}.b$ with $b\in A_+$. In particular $\Phi_{w\gl}(a)\sim e^{w\gl-\rho}(\textup{log}(a))$ when $a\to \infty$ in the 
cone $A_+.o$. 

 Consequently, the spherical function $\varphi_\gl$ admits a unique expansion on $A_+.o$ of the form:  
\begin{equation}\label{eq:exp_phi}
\varphi_\gl(a.o) = \sum_{w\in W}c(w\gl)\Phi_{w\gl}(a.o)
\end{equation}
The important function $c(\gl)$ (known as {\it Harish-Chandra's $c$--function}) is uniquely determined by (\ref{eq:exp_phi}),  
and is meromorphic on $\fa^*_c$.

From general principles, Harish--Chandra expected from this asymptotic behaviour of $\varphi_\gl$ that the Plancherel measure for the 
spherical Fourier transform would be of the form (with $\gl\in\fa^*$):
\begin{equation}
d\mu_{Pl}^{sph}(i\gl):=\frac{1}{|W||c(i\gl)|^2}d\gl, 
\end{equation}
where $i\gl\in i\fa^*$, and where $d\gl$ defines the measure on $i\fa^*$ which is Fourier dual to the measure $da$ on the vector group $A$.  
Equivalently, for spherical functions $f\in L^2(X,dg)^K$ which are suitably well behaved, he expected that:
\begin{equation}\label{eq:Plan}
|W|f(o)=\int_{\gl\in \fa^*}\Cf^G_{sph}(f)(i\gl)|c(i\gl)|^{-2}d\gl
\end{equation}    

In \cite{HC3}, \cite{HC3II} Harish--Chandra discovers important new notions and subtleties, but he is also confronted with technical obstacles. 
He introduces a very important commutative convolution algebra of smooth spherical functions on $G$ for which (\ref{eq:Plan}) should hold, the spherical part 
$\Cs(G/\!\!/K)$ of the Harish--Chandra Schwartz space of $G$. This is based on an upper and a lower bound for the function $\varphi_0$. The elementary spherical 
functions themselves are novel and interesting objects.
In the complex case the elementary spherical functions $\varphi_\gl$ can be expressed in terms of elementary functions, but the general rank one case showed 
that in general, $\varphi_\gl$ 
is a transcendental function. In the rank one case he found that $\varphi_\gl$ can be expressed as a Gauss hypergeometric function 
${}_2F_1(a,b;c;z)$, with parameters $a,b,c$ depending on the root multiplicities $k_{2\ga}:=\frac{1}{2}m_{\ga}=\frac{1}{2}\textup{dim}(\fg_\ga)$ 
(with $2\ga\in  R=2\Phi$). 
Harish--Chandra \cite{HC3} showed that for $\textup{Re}(\gl)\in\fa^*_+$ and $H\in\fa_+$ (write $a=\textup{exp}(H)\in A_+$): 
\begin{equation}\label{eq:lim}
c(\gl)=\lim_{t\to\infty}e^{(\rho-\gl)(tH)}\varphi_\gl(a^t) = \int_{\overline{N}}e^{-(\gl+\gd)(H(\overline{n}))}d\overline{n}
\end{equation}

In view of (\ref{eq:lim}), the $c$--function is a "connection coefficient" for the system of differential equations (\ref{eq:spheq}).  
Harish--Chandra computed this all important $c$--function for the rank one cases using the theory of hypergeometric functions. 
Also he computed this function in the complex case, in which case it is an elementary function. 
However, a general explicit expression remained elusive at this point, and this would only be resolved later, by the work of 
Gindikin and Karpelevi\v c \cite{GiKaI,GiKaII}. A representation theoretic interpretation of the Harish--Chandra $c$--function 
in terms of intertwining operators between the minimal principal series representations of $G$ was provided by Schiffmann \cite{Sch},  
also see \cite{Knapp}. 

In \cite{HC3II} Harish--Chandra proved the inversion formula (\ref{eq:Plan}) modulo two conjectures. The first conjecture states
that there exists a polynomial upper bound for $|c(i\gl)|^{-2}$ for $\gl\in \fa^*$. This is solved easily by the use of Stirling's Theorem
in combination with the explicit formula for the $c$--function in \cite{GiKaI,GiKaII}.  The second conjecture is the injectivity
of the Abel transform $F_f$ for $f$ in $\Cs(G/\!\!/K)$. This was solved by Harish--Chandra in \cite{HC5}, thereby finally completing
his proof of the spherical Plancherel formula for the spherical Schwartz space $\Cs(G/\!\!/K)$. 

Another approach to the Plancherel formula of $G/K$ was initiated by Ehrenpreis and Mautner \cite{ErMa}, and followed up by Takahashi,  
Helgason, Gangolli and finally Rosenberg \cite{Ros} (see \cite[IV, Notes]{Helgason} for an account of this development and further references). 
This approach is based on a Paley--Wiener Theorem, and simplified the proof of the Plancherel formula for $G/K$, see \cite[Ch. IV]{Helgason}. 
The image of the spherical Fourier transform of the space $C_c^\infty(G/\!\!/K)$ is contained in the classical Paley--Wiener space for the abelian group 
$A$ by Propositions \ref{prop:ImAb} and \ref{prop:transmut}, using the properties of $F_f$. The proof of the converse statement is 
remarkable and direct, following Rosenberg \cite{Ros}.

Finally we arrive at the following results about the "transmutation operator," the Abel transform $F_f$:
\begin{cor}\label{cor:ATCcandS}
$F_f:C^\infty_c(G/\!\!/K)\to C_c^\infty(A)^W$ is an algebra isomorphism. Similarly, if $\Cs(G/\!\!/K)$ denotes
the space of spherical functions of the Harish--Chandra Schwartz space of $G$, and $\Cs(A)^W$ the space
of Weyl group invariant elements of the Schwartz space of $A$, $F_f:\Cs(G/\!\!/K)\to \Cs(A)^W$ is an isomorphism 
of algebras. 
\end{cor}

\section{Satake's isomorphism and its inversion}
\label{sec:SI}

We will make an excursion into the study of spherical Harmonic analysis on reductive groups of $p$--adic
type. There are remarkable analogies, but importantly, also some lessons to be learnt. The Iwahori--Hecke
algebra makes its presence felt, even when studying $K$--spherical functions for a maximal compact subgroup
$K$. In the real case this is in some sense also true, but it is much more hidden. 

\subsection{The spherical Hecke algebra of a $p$--adic group}

Let $\sfk$ be a non--Archimedean local field, and let $G$ be the group of $\sfk$-points of a reductive group
$\underline{G}$ defined over $\sfk$. Assume that $\underline{G}$ is not anisotropic over $\sfk$, and let $A
\subset G$ be (the group of $\sfk$--points of) a maximal split $\sfk$-torus. Let $\Cb$ be the Bruhat--Tits
building of $G$, and $\Ca\subset\Cb$ the apartment corresponding to $A$. Fix a special point $x_0\in\Ca$,
and let $R$ the corresponding finite root system. We choose a set $R_+$ of positive roots, with corresponding
positive Weyl chamber $\Cc$, and let $\Cc_0\subset \Cc$ be the alcove in $\Cc$ containing $x_0$. There is
a canonical affine root $ R^{a}$ system associated with $\Ca$, with affine Weyl group $W^a$ and extended
affine Weyl group $W^{e}=W^a\rtimes \Omega$ acting on $\Ca$, with $\Omega=\{w\in W^{e}\mid w(\Cc_0)
=\Cc_0\}$.

Put $K=G_{x_0}$. The lattice $L=A/(A\cap K)$ comes equipped with a natural $W$--action, and acts on $\Ca$
by translations. We have a splitting $L\to A$ which we will denote by $L\ni\lambda\to \pi^\lambda\in A$. It is
well--known that $\Omega$ is an abelian group, isomorphic to the quotient of $L$ by the coroot lattice $\bbz
R^\vee$.

Let $L_+\subset L$ be the monoid of positive translations relative to the chamber $\Cc$. We have the Cartan decomposition
$G=KL_+K$. The convolution algebra $C_c^\infty(G/\!\!/K)$ of compactly supported, $K$--biinvariant locally constant functions
on $G$ is called the {\it spherical Hecke algebra} of $(G,K)$ and denoted by $\mathbf{H}(G,K)$. Its unit element $e_K$ is
the characteristic function of the double coset $KeK$. One can show - in a similar fashion to the real case - that $\bH(G,K)$
is a commutative algebra. 

Analogously to the real case, we say that an element $\go\in C^\infty(G/\!\!/K)$ is an elementary (or zonal) spherical
function $K$--spherical function if $\go$ is a joint eigenfunction of $\mathbf{H}(G,K)$. As explained in Section \ref{ss:D spherical},
in the real case zonal spherical functions can also be characterised as joint $K$-bi-invariant eigenfunctions $\omega$
of the algebra $\bbd(X)$ of invariant differential operators on $X=G/K$ such that $\omega(e)=1$, and are parametrised
by $\fa^*_c/W$.

\subsection{The Satake transform}
Satake, inspired by these results of Harish--Chandra, set out to find a similar parametrization of the set of zonal 
spherical functions in the $p$--adic case \cite{S0,S}. Although in the $p$--adic setting the "infinitesimal version" $\bbd(X)$ of
the spherical Hecke algebra is not present, Satake realized that we can do without this here, because $\bH(G,K)$
is itself a finitely generated commutative $\bbc$--algebra. 

Using the Iwasawa decomposition $G=KAN$, Satake first constructs, for every complex character 
\begin{equation}
s\in T:=\textup{Hom}(L,\bbc^\times)
\end{equation} 
of $L$, 
a zonal spherical function $\go_s$ by analogy with (\ref{eq:phiHC}). The spherical Fourier transform of $f\in \mathbf{H}(G,K)$ 
at $s\in T$ is defined by:  
\begin{equation}\label{eq:sphft}
\Cf^G_{sph}(f)(s):=(\go_s,f^*)=\int_{g\in G}f(g)\go_s(g^{-1})dg, 
\end{equation}
By analogy with  (\ref{eq:abel}), he introduces a homomorphism  
\begin{align}\label{eq:sat}
\gamma_\Cs: \mathbf{H}(G,K)&\to \mathbf{H}(A,A\cap K)^W=\bbc[L]^W\\
\nonumber f&\to \left\{a\to \delta^{1/2}(a)\int_N f(an)dn\right\}
\end{align}
where $dn$ is the Haar measure on $N$, normalized so that the volume of $N\cap K$ is $1$, 
and $\delta$ denotes the Haar modulus of the group $AN$. 

As in the real case, we have: 
\begin{equation}\label{eq:abft}
\Cf^G_{sph}(f)(s)=\Cf^A_{sph}(\gc_\Cs(f))(s)
\end{equation}
where $\Cf^A_{sph}$ denotes the spherical Fourier transform on the abelian group $A$ relative to the 
compact subgroup $A\cap K$.  
Satake proves that 
\begin{equation*}
\Cf^G_{sph}: \mathbf{H}(G,K)\to \bbc[T]^W
\end{equation*}
is an algebra isomorphism, or equivalently (via the Fourier transform on $L$):  
\begin{thm}[\cite{S}]
Let $L=A/(A\cap K)$, a lattice of rank $n=\textup{dim}(\mathcal{A}_A)$. The above map 
\[\gamma_S: \mathbf{H}(G,K)\to \mathbf{H}(A,A\cap K)^W=\bbc[L]^W\]
is an algebra isomorphism. In particular, $\mathbf{H}(G,K)$ is isomorphic to $\bbc[T]^W$ (a subalgebra
of finite index of a polynomial algebra with $n$ generators).
\end{thm}
\begin{cor}
The set of zonal spherical functions is naturally parametrized by the set of points of the affine variety
${W\backslash T}$ over $\bbc$.
\end{cor}
A major difference with the real case is the fact that the zonal spherical functions in the $p$--adic case are 
{\it rational} functions, which can be made \emph{completely explicit}. 
Recall the integral representation for zonal spherical functions (which is the $p$--adic analog of (\ref{eq:phiHC})).
Given $s\in T$ we define the map $\nu_s:G\to\bbc^\times$, using the Iwasawa decomposition $G=KAN$, by 
$\nu_s(g)=\nu_s(kan)=(\delta^{-1/2}s)(a)$. Then we define:
\begin{equation}
\go_s(g):=\int_K\nu_s(g^{-1}k)dk
\end{equation}
In the $p$--adic case this integral was evaluated directly by I. G. Macdonald (\cite{Ma2}), thereby also explicitly
inverting Satake's isomorphism $\gc_S$ and computing the spherical Plancherel measure (see Theorem \ref
{thm:mac}).

\subsection{The Iwahori subgroup and Iwahori--Hecke algebra $\bH(G,B)$}

With the benefit of hindsight, these developments are best understood by placing them in the
context of another major insight that arose in the representation theory of $p$--adic reductive
groups. Let $B\subset K$ be the Iwahori subgroup fixing the alcove $\Cc_0$. 
Casselman and Borel \cite{Cas2,Bo} pointed out that the category $\mathcal{C}(G)_{(B,1)}$
of B--spherical smooth representations is an {\it abelian} subcategory of the category $\mathcal{C}(G)$
of smooth representations of $G$, and that it is generated by the unramified minimal principal series. 
This is in stark contrast with the category of $K$--spherical smooth representations, which is not abelian.

Let
\[ \bH(G,B)=C_c^\infty(G/\!\!/B)\]
be the {\it Iwahori--Hecke} algebra of $(G,B)$. Then, the functor $\mathcal{C}(G)\ni V\to V^B$ restricts
to an equivalence of $\mathcal{C}(G)_{(B,1)}$ with the category of $\bH(G,B)$--modules.

The spherical Hecke algebra $\bH(G,K)$ naturally embeds in $\bH(G,B)$, though not as a unital subalgebra
since their unit elements $e_{K},e_{B}$ are different. The computation of the zonal spherical functions and
their asymptotic expansion can be reduced to an elementary study of the $\bH(K,B)$--spherical representation
theory of $\bH(G,B)$ (where the trivial representation of $\bH(K,B)$ is the one dimensional representation
$T_w\to q(w)$ of $\bH(K,B)$). 

\subsection{Two presentations of $\bH(G,B)$}

The affine roots $a\in R^a$ of $G$ correspond to root subgroups $N_a$ such that $N_a\subset N_{a-1}$, 
with finite index $q_a:=(N_{a-1}:N_a)$.
Macdonald \cite{Ma2} defines a possibly non--reduced root system $R_{nr}$ containing $R$, by 
adding $\ga/2$ as a root whenever the affine roots $q_\ga$ and $q_{\ga+1}$ are distinct, and defined:\footnote
{Macdonald operates in the context of simply connected Chevalley groups, but his setup can be generalized to
reductive groups, e.g. using the theory developed in \cite{I}.}
\begin{equation}
q_{\ga/2}:=q_{a+1}/q_a
\end{equation}
Let $V=N_G(A)$ be the normalizer of $A$ in $G$, then $(V,B)$ is an extended (in the sense of \cite{I}) Tits pair,
with Weyl group the extended affine Weyl group $V/(A\cap B)=:W^e$.  
This extended affine Weyl group satisfies $W^e=L\rtimes W$, 
where $W=W_{x_0}$. The length function $l$ on the affine Weyl group $W^a$  can be extended to $W^e$ (cf. \cite[Section 1]{EO}), and 
the set $\Omega\subset W^e$ of elements of length $0$ is a normal abelian subgroup such that 
$W^e=W^a\rtimes \Omega$. 
For elements $w\in W^e$ of the extended affine Weyl group, one defines 
\begin{equation}
q(w):=(Bn_wB:B)=\prod_{a\in  R^a_+\cap w^{-1} R^a_-}q_a
\end{equation}
where $n_w$ is a representative of $w$ in the normalizer $V$ of $A\cap B$ in $G$.  

The Iwahori--Hecke algebra $\bH(G,B)$ has a $\bbc$-linear basis $\{T_w\}_{w\in W^e}$, where $T_w$ is the characteristic 
function of the double coset $Bn_wB$, divided by the volume of $B$. Let $S^a$ denote the set of affine simple reflections in
$W^e$. Basic facts of Bruhat--Tits theory applied to the Tits pair $(B,N)$ then yields the well--known relations in $\bH(G,B)$ as 
in \cite{IM}, \cite{I}: 
\begin{itemize}
\item[(a)] If ${l}(ww^\prime)={l}(w)+
{l}(w^\prime)$ then $T_wT_{w^\prime}
=T_{ww^\prime}$.
\item[(b)] If $s\in S^a$ then
$(T_s+1)(T_s-q(s))=0$.
\end{itemize}

The Iwahori--Hecke algebra $\bH(G,B)$ admits a trace $\tau$ given by:
\begin{equation}
\tau(\sum_{w\in W^e}c_wT_w)=c_e
\end{equation} 
and anti-linear 
anti--involution $*$ such that $T_w^*=T_{w^{-1}}$, giving rise to a positive definite Hermitean form $(a,b):=\tau(ab^*)$ 
such that $(T_v,T_w):=\delta_{v,w}q(w)$.   

Given $\gl\in L$ write $t_\gl\in W^e$ for the corresponding translation, viewed as an element of the extended affine Weyl group $W^e$. 
Bernstein defined an algebra embedding 
\begin{equation}\label{eq:bern}
\eta_B: \mathbf{H}(A,A\cap B)=\bbc[L]\xhookrightarrow{} \mathbf{H}(G,B), 
\end{equation}
characterised by: If $\lambda\in L_+\subset L=A/(A\cap B)$, the commutative monoid of positive 
translations of $\mathcal{A}_S$, then:  
\begin{equation*}
\eta_B(\pi^\gl e_{A\cap B})=q(t_\gl)^{-1/2}T_{t_\lambda}:=q(t_\gl)^{-1/2}T_e \pi^\lambda T_e\in \mathbf{H}(G,B).
\end{equation*}

This gives rise to another presentation of the algebra $\bH(G,B)$, the Bernstein--Zelevinski
presentation.\footnote{This is unpublished work of Bernstein and Zelevinski, but see e.g. \cite{Lu}}
For $\gl\in L$ we put $\theta^A_\gl=\pi^\gl e_{A\cap B}$ and we define $\theta_\gl:=\eta_B(\theta^A_\gl)\in\bH(G,B)$. Let $\bA\subset \bH(G,B)$ denote 
the the image of $\eta_B$ (the commutative subalgebra of $\bH(G,B)$ generated by the $\theta_\gl$). 
\begin{thm}\label{thm:bzbasis}\hfill
\begin{itemize}
\item[(a)] The elements $\{T_w\theta_\gl\}_{w\in W,\ \gl\in L}$ are
a basis of $\bH(G,B)$.
\item[(b)] The elements $\{\theta_\gl T_w\}_{w\in W,\ \gl\in L}$ are also 
a basis of $\bH(G,B)$.
\item[(c)] 
Let $\gl\in L$, $\alpha$ a simple root for $ R$, and set $s=s_\alpha$. Then
\begin{gather*}
\begin{split}
\theta_\gl&T_s-T_s\theta_{s(\gl)}=\\
&=\left\{
\begin{array}{ccc}
&(q_{\alpha}-1)\frac{\theta_\gl-\theta_{s(\gl)}}
{1-\theta_{-\alpha^\vee}}\ &{\rm if}\ \alpha/2\not\in R_{nr}.\\
&((q_{\alpha/2}q_{\alpha}-1)
+q_{\alpha/2}^{1/2}
(q_{\alpha}-1)\theta_{-\alpha^\vee})
\frac{\theta_\gl-\theta_{s(\gl)}}
{1-\theta_{-2\alpha^\vee}}\ &{\rm if}\ \alpha/2\in R_{nr}.\\
\end{array}
\right.\\
\end{split}
\end{gather*}
\end{itemize}
\end{thm}

One important consequence of this theorem is the precise description of the center $\bfz$ of $\bH(G,B)$.
The following result is also due to Bernstein and Zelevinski.
\begin{thm}\label{thm:cent}\hfill
\begin{enumerate}
\item $A$ and $A\cap B$ are $W$--invariant, hence $\bH(A,A\cap B)$ is naturally a $W$--module
via the action of $W$ on $L=A/(A\cap B)$ (i.e. $w(\theta_\gl^A)=\theta^A_{w\gl}$).
\item
Transfer the $W$--action to $\bA$ via $\eta_B$. 
Then the center $\bfz=\bfz(\bH(G,B))$ equals
\[\bfz=\eta_B(\bH(A,A\cap B)^{W})=\bA^W\]
In particular, $\eta_B$ restricts to an isomorphism from $\bH(A,A\cap B)^W$ onto $\bfz$
\end{enumerate}
\end{thm}
The following result is also well--known.
\begin{proposition}\label{prop:star}
Let ${w_0}\in W$ denote the longest element, and $\gl\in L$. Then:
 \[
\theta^*_\gl=T_{w_0}\theta_{-w_0({\gl})}T_{w_0}^{-1}.
\]
\end{proposition}

\subsection{Macdonald's explicit formula for spherical functions}

Let $e_K\in \bH(G,B)$ be the unit of the spherical Hecke algebra. Then obviously we have
$e_K=P_W(q)^{-1}\sum_{w\in W} T_w$ where $P_W(q):=\sum_{w\in W}q(w)$ denotes the
Poincar\' e polynomial of the Weyl group $W$. Then $\mathbf{H}(G,K)=e_K\mathbf{H}(G,B)
e_K\subset \mathbf{H}(G,B)$.

By the Cartan decomposition, $\mathbf{H}(G,K)=e_K\eta_B(\bbc[L])e_K$ (indeed, the set
$\{e_K \theta_\gl e_K\mid \gl\in L_+\}$ is a basis). Using Theorem \ref{thm:bzbasis}, one
can prove easily (cf. \cite[Lemma 2.28, Theorem 2.29]{EO}, \cite[equation (1.23)]{HO}) 
that for all $\gl\in L$: 
\begin{equation}\label{eq:mac}
\theta_\gl^+:=e_K\theta_\gl e_K=e_KP_W(q^{-1})^{-1}\eta_B(\sum_{w\in W} w(c(\cdot,q)\theta^A_\gl))
\end{equation} 
where 
\begin{equation*}
c(\cdot,q)=\prod_{\alpha\in  R_{nr,+}}\frac{1-q^{-1}_\alpha q^{-1}_{\alpha/2}\theta^A_{-\alpha^\vee}}{1-q^{-1}_{\alpha/2}\theta^A_{-\alpha^\vee}}
\end{equation*} 
(here the $\cdot$ refers to the interpretation of this expression as a rational function on $T$).
\begin{cor}\label{eq:sphcent}
We have $\bH(G,K)=e_K\bfz$, and 
\[\epsilon:\bH(A,A\cap B)^W\to \bH(G,K),
\qquad
z\to e_K\eta_B(z)\]
 is 
an isomorphism. 
\end{cor}
\begin{proof}
By (\ref{eq:mac}) and Theorem \ref{thm:cent} we see that $\bH(G,K)\subset e_K\bfz$, while the opposite inclusion is obvious.
By Theorem \ref{thm:cent} we have $\bfz=\eta_B(\bH(A,A\cap B)^W)\subset \bA$, and $e_K=P_W(q)^{-1}\sum_{w\in W} T_w$. Hence $\gep$ is an 
isomorphism by Theorem \ref{thm:bzbasis} and (\ref{eq:bern}). 
\end{proof}

Let $t\in W\backslash T$. Let $\bA:=\eta_B(\bH(A,A\cap B))\simeq \bbc[L]\subset \bH(G,B)$ and consider the induced representation  
\begin{equation}
\pi_t=\textup{Ind}_\bA^{\bH(G,K)}\bbc_t=\bH(G,B)\otimes _\bA\bbc_t=\bH(K,B)\otimes\bbc_t. 
\end{equation}
This defines a family of $\bH(K,B)$--spherical
representations with one dimensional space of spherical vectors generated by $e_K\otimes \bbc_t$.  
\begin{df}
Define the functional 
$\go_t\in\textup{Hom}_{alg}(\bH(G,K),\bbc)$ by
\begin{equation}\label{eq:trace}
\tau(\go_t f)=(\go_t,f^*):=\textup{Tr}(\pi_t(f))
\end{equation}  
for any $f\in\bH(G,K)$ (cf. \cite[Section 2]{HO}).
\end{df}
Then $\go_t$ is a linear functional on $\bH(G,K)$, and can be expressed as a formal sum of the elements $\theta_\gl^+$ 
(with $\gl\in L_+$). 
Since the space of spherical vectors in $\textup{Ind}_\bA^{\bH(G,K)}\bbc_t$ is one dimensional, $\go_t$ is an eigenfunction of $\bH(G,K)$, 
and thus a zonal spherical function. Indeed, we see that for all $f\in\bH(G,K)$: 
\begin{equation}\label{eq:satiso}
f\go_t=(\go_t,f^*)\go_t
\end{equation}
In fact, $\go_t$ is the zonal spherical function as in (\ref{eq:sphft}), which Satake associated to $t$. 

Macdonald's Theorem on the explicit form of $ \go_t$ is now a simple consequence of (\ref{eq:mac}):
\begin{thm}\label{thm:mac}
For $\gl\in L_+$ we have 
\begin{equation}\label{eq:sphfu}
(\go_t,\theta^*_\gl)=\frac{1}{P_W(q^{-1})}\sum_{w\in W}c(wt,q)wt(\gl)
\end{equation}
or equivalently
\begin{equation}\label{eq:macsway}
\go_t(\pi^{-\gl})=\frac{\delta^{-1/2}(\pi^\gl)}{P_W(q^{-1})}\sum_{w\in W}c(wt,q)wt(\gl)
\end{equation}
\end{thm}
\begin{proof}
Using (\ref{eq:mac}) we see easily that (\ref{eq:sphfu}) follows from the definition of $ \go_t$ as in (\ref{eq:trace}).
Macdonald's formula \cite[Thm 4.1.2]{Ma2} follows from the remark that the characteristic function of the double 
orbit $K\pi^\gl K$ is given by $\chi_{K\pi^\gl K}:=\frac{\delta^{1/2}(\pi^\gl) \theta^+_\gl}{(\theta^+_\gl,\theta^+_\gl)}$ (cf. \cite[Proposition 2.32]{EO} and its proof).
Hence we can write $ \go_t=\sum_{\gl\in L_+} \go_t(\pi^\gl)\chi_{K\pi^\gl K}$, and 
\begin{align*}
(\go_t,\theta_\gl^*)&=(\go_t,\theta_{-w_0\gl}^+)\\
&=\go_t(\pi^{-w_0\gl})\delta^{1/2}(\pi^\gl)\\
&=\go_t(\pi^{-\gl})\delta^{1/2}(\pi^\gl)\\
&=\frac{1}{P_W(q^{-1})}\sum_{w\in W}c(wt,q)wt(\gl)
\end{align*}
and the result for $\go_t(\pi^{-\gl})$ follows easily (this proof is in spirit an algebraic version of Casselman's 
proof in \cite{Cas2}).
\end{proof}
Satake's result on $\gc_S$ follows too, and we even find that $\gc_S$ is a very simple and explicit isomorphism:
\begin{thm}\label{thm:macsat}
The inverse of $\gc_S$ (cf. (\ref{eq:sat})) is the isomorphism $\gep$ of Corollary \ref{eq:sphcent}. 
\end{thm}
\begin{proof}
It is enough to prove this for the basis elements $\theta_\gl^+$ ($\gl\in L_+$) of $f\in\bH(G,K)$. 
By equation (\ref{eq:abft}) and (\ref{eq:satiso}), we see that $\Cf^A_{sph}(\gc_S(\theta^+_\gl))(t)=\Cf^G_{sph}(\theta^+_\gl)(t)=(\go_t,\theta_\gl^*)$, 
so Theorem \ref{thm:mac} implies that: 
\begin{equation}\label{eq:satinv}
\gc_S(\theta^+_\gl)=\frac{1}{P_W(q^{-1})}\sum_{w\in W}w(c(\cdot,q)\theta^A_\gl)
\end{equation}
We conclude that $\gep(\gc_S)=\textup{id}$ by (\ref{eq:mac}). 
\end{proof}
\begin{remark}\label{rem:extber}
Observe that Bernstein's embedding $\eta_B$ can be extended to an isomorphism $\eta_B:\mathbf{H}(L,W;q)\to \mathbf{H}(G,B)$ 
of affine Hecke algebras, where $W^e=L\rtimes W$ (the Bernstein--Zelevinsky presentation  of $\mathbf{H}(G,B)$).
\end{remark}
Next we will consider a real analog of all of this. 

\section{Graded Hecke algebras and Dunkl operators}


We return to the world of real reductive groups and Riemannian symmetric spaces. We will introduce in some
sense an extended Bernstein isomorphism in this context, mapping a {\it graded affine} Hecke algebra to certain
differential {\it reflection} operators on $A\subset \textup{exp}(\fp)$. 

In Section \ref{sec:HCH}, we considered a
non--compact Riemannian symmetric space $X=G/K$ with Cartan decomposition $G=KAK$, and extracted
several data and structures on $A$ which capture the geometry of $X$.
In the present Section, we will start
more generally with the introduction of such structures on a real vector group $A$, but not necessarily coming
from a symmetric space $X$. 

Set {$\fa=\textup{Lie}(A)$}, and equip $\fa$ with a Euclidean structure. Let {$R\subset \fa^*$} be a (possibly
non--reduced) root system, and $R^\vee\subset \fa$ the coroot system. Let $Q=\bbz R\subset \fa^*$ be the
root lattice, and $Q^\vee=\bbz R^\vee\subset \fa$ the coroot lattice with dual weight lattice $Q\subset P=
\Hom(Q^\vee,\bbz)\subset \fa^*$. Let $T_c=\Hom(P,\bbc^\times)\simeq \ft/2\pi i Q^\vee$ be the complex
algebraic torus with character lattice $P$ and Lie algebra $\mathfrak{t}:=\bbc\otimes_\bbr\fa$. $T_c=T_uA$
for the polar decomposition, where 
\begin{equation}\label{eq:polar T}
T_u=\Hom(P,S^1)\aand A=\Hom(P,\bbr_{>0})
\end{equation}
are the corresponding compact torus and real vector group respectively.

\begin{remark}
Let us explain the comparison with the case of a Riemannian symmetric space $X=G/K$. 
Here we think of $A$ as a maximal flat subspace of $X$, 
not as the maximal split Cartan $A_G=\textup{exp}(\fa)\subset \textup{exp}(\fp)\subset G$.  
In the complexification $G_c$ of $G$, $A_{G,c}\cap K_c$ is the $2$-torsion subgroup of $A_{G,c}$ (the elements $a\in A_{G,c}$ 
such that $\theta(a)=a=a^{-1}$, where $\theta$ is the (holomorphic) Cartan involution), hence the isogeny $A_{G,c}\to A_{G,c}$, $a\to a^2$ 
descends to an isomorphism $T_c=A_{G,c}/(A_{G,c}\cap K_c)\to A_{G,c}$. Therefore, the root system $R\subset \fa^*$ identifies with $R=2\Phi(\fg,\fa)$. 
Also observe that the natural isomorphism from $A$ to $A_G$ is given by $a\to a^2$, and this is the map we use to identity these two vector groups. 
Note that $P=2P_\Phi$ and $A_{G,c}=\textup{Hom}(P_\Phi,\bbc^\times)$. 
\end{remark}

For {$\mu \in P$}, the {corresponding character} in $\Hom(T,\bbc^\times)$ is denoted by {$t^{\mu}\in \Hom(T,\bbc^\times)$}.
The ring of regular functions $\bbc[T]$ on $T$ is isomorphic to the group algebra $\bbc[P]$ of $P$ via the Fourier transform, which we often use to identify 
these two with the group algebras.  The Weyl group $W$ of $R$ acts on everything in sight. 
Choose a Weyl chamber $\fa^*_+$, let $\{s_1,\dots , s_n\}\subset W$ be the corresponding set of simple reflections.
Put $P_+:=P\cap \overline{\fa^*_+}$, the monoid of dominant weights.  

\subsection{Dunkl--Cherednik operators on $\bbc[T]$}
Let 
\begin{align}
\mathcal{K}&=\left\{k=\{k_{\alpha}\}\in\bbc^R\mid k_{w\alpha}=k_{\alpha},\,w \in W,\alpha \in R\right\}
\label{eq:cK}\\
\Ck_+&=\{k\in\Ck\mid k_\ga\geq0\,\,\alpha \in R\}
\end{align}
be the space of {root multiplicity parameters} and positive root multiplicity parameters, respectively.

Let $\Cd(T^{reg})[W]$ be the ring of {\it differential--reflection} operators on $T^{reg}$. Following I.
Cherednik \cite{C}, we define the operators $T(\xi,k)\in\Cd(T^{reg})[W]$, $\xi\in\ft$, $k\in\mathcal{K}$
by
\[T(\xi,k)
=\partial(\xi)+\sum_{\alpha>0}k_{\alpha}\alpha(\xi)(1-t^{-\alpha})^{-1}(1-s_{\alpha})-\rho(k)(\xi)\]
where $\partial(\xi)t^{\mu}=\mu(\xi)t^{\mu}$, and $\rho(k)\in\ft^*$ is given by $\tfrac12 \sum_{\alpha
>0}k_{\alpha}\alpha\in\mathfrak{t}^{\ast}$.
Despite the denominators in the expression of $T(\xi,k)$, we have that $T(\xi,k)|_{\bbc[T]}\in\textup
{End}(\bbc[T])$ and $T(\xi,k)|_{C^\infty_c(A)}\in\textup{End}(C^\infty_c(A))$. A crucial property is: 
\begin{thm}(I. Cherednik \cite{C}, G.J. Heckman \cite{OGHA})
The operators $T(\xi,k)$ commute: for all $\xi,\eta\in\ft$ we have $[T(\xi,k),T(\eta,k)]=0$.
\end{thm}

\subsection{Hypergeometric differential operators and $\bbd(X)$}
Using the Dunkl--Cherednik operators one can construct $W$--invariant {\it differential}
operators, the so-called hypergeometric operators. Extend the linear map $\xi\to T(\xi,k)$
to an algebra homomorphism
\begin{align*}
\eta_B: S(\ft)&\to \End(\bbc[T])\\
p&\to \eta_B(p)=:T(p,k)
\end{align*} 
Let $\mathbb{A}(W\backslash T)$ be the Weyl algebra of polynomial differential operators on the affine 
space $W\backslash T$. Then, the following holds

\begin{thm}
If $p\in S(\ft)^W$ then $T(p,k)$ is $W$--invariant. Its restriction to $\bbc[T]^W$ is a differential operator
$D(p,k)\in \mathbb{A}(W\backslash T)$ .
\end{thm}

The {\it hypergeometric differential operators} $D(p,k)$ were originally constructed in \cite{O4}, prior to to
the introduction of the operators $T(\xi,k)$. Whenever the parameters $k_\ga$ arise from a Riemannian
symmetric space $X$, they are the radial parts of the operators in $\bbd(X)$: 
\begin{thm}\label{thm:RSS}
 Let $X=G/K$ be a Riemannian symmetric space, and $\fa\subset \fp$ maximal abelian.  
Put {$R=2\Phi(\fg,\fa)$} and {$k_{\alpha}=\textup{dim}(\fg_{\alpha/2})/2$}.
For $p\in \bbc[\fa^*]^W$ recall that $\gamma_{HC}^{-1}(p)\in\bbd(X)$ is the $G$--invariant 
operator on $X$ corresponding to $p$ via the {Harish--Chandra homomorphism}. 
Then (analogously to $\gamma_{S}^{-1}(f)=\eta_B(f)e_K$):
\begin{equation*}
\textup{Rad}(\gamma_{HC}^{-1}(p))=\eta_B(p)|_{C^\infty(A)^W}{(=D(p,k))\in\mathbb{A}(W\backslash T).}
\end{equation*}
were {$\textup{Rad}(D)$} (with $D\in\bbd(X)$) denotes the {radial part} of $D$ on $A$.
\end{thm}
The result of Theorem \ref{thm:RSS} should be compared with Theorem \ref{thm:macsat} for $p$--adic groups.
But notice that something is not quite right: The idempotent $e_K$ seems to be missing in the real version. 
This will be fixed shortly, the problem is caused by the the fact that $\bbd(X)$ is not quite the correct analog
of the spherical Hecke algebra $\bH(G,K)$. Rather, we should be looking at the spherical algebra 
of the Archimedean Hecke algebra $\ffR(\fg,K)$ which is isomorphic to $\bbd(X)$ but not quite the same. 
See Subsection \ref{sub:comp} for more details.

\subsection{The graded affine Hecke algebra}
We saw in Remark \ref{rem:extber} that $\eta_B$ in the $p$--adic version could naturally be extended to
an algebra homomorphism from an abstract affine Hecke algebra given in its Bernstein presentation to the
concrete Iwahori--Hecke algebra of $G$. In the present situation we do not have an "Iwahori--Hecke algebra"
available, so this would be too much to ask for. But we do have something to offer which - at least partially - 
restores the analogy. We will discuss this in the next two sections.
\begin{df}
The {graded affine Hecke algebra $\mathbb{H}=\mathbb{H}(R_+,k)$} is the unique associative algebra satisfying
\begin{itemize}
\item[(i)] $\mathbb{H}=\bbh(R_+,k):=S\mathfrak{t}\otimes\bbc[W]$ as a vector space over $\bbc$,
\item[(ii)] $S\mathfrak{t}\rightarrow\mathbb{H},\;p \mapsto p\otimes1$ and $\bbc[W]\rightarrow\mathbb{H},\;w\mapsto1\otimes w$
are algebra homomorphisms, and so we will identify $S\mathfrak{t}$ and $\bbc[W]$ with their images in $\mathbb{H}$ via these maps,
\item[(iii)] $p\cdot w=p\otimes w$ with $\cdot$ denoting the algebra multiplication in $\mathbb{H}$,
\item[(iv)] $s_i\cdot p-s_i(p)\cdot s_i=-k_i(p-s_i(p))/\alpha_i^{\vee}$ with w(p) the natural transform of $p\in S\mathfrak{t}$
under $w\in W$, and $k_i=\tfrac{1}{2}k_{\alpha_i/2}+k_{\alpha_i}$.
\end{itemize}
\end{df}
\begin{cor} 
The center {$Z(\mathbb{H})$} of $\mathbb{H}$ is isomorphic to {$S\mathfrak{t}^W$}.
In particular, $\bbh$ is finite over $Z(\bbh)$ and all its irreducible representations are
finite--dimensional.
\end{cor}
\begin{remark}\label{rem:sphsa}
Let
\[\gep_+=\frac{1}{|W|}\sum_{w\in W}w\in \bbh(R_+,k)\]
be the idempotent corresponding to the trivial representation of $W$. The {\it spherical algebra}
$\bbh^{sph}(R_+,k)\subset \bbh(R_+,k)$ is the algebra $\gep_+ \bbh(R_+,k)\gep_+\subset\bbh
(R_+,k)$ (with unit element $\gep_+$). Then $\bbh^{sph}(R_+,k)=\gep_+ Z(\bbh(R_+,k))$, and
in particular $\bbh^{sph}(R_+,k)$ is commutative.
\end{remark}

\subsection{Dunkl--Cherednik representation $\eta$ and symmetric spaces}

The algebraic (or compact) version of the Dunkl--Cherednik representation is the analog
of the (extended) Bernstein isomorphism we encountered in the $p$--adic case.

\begin{thm}\label{thm:etau}(\cite{C}, \cite{He91a}, \cite{He91}, \cite{O4}, \cite{OGHA})
There is a faithful representation  $\eta_u=\eta_u(k):\bbh\to \textup{End}(\bbc[T])$ (the
Dunkl--Cherednik representation), defined by 
\begin{itemize}
\item[(a)] For $p\in S(\ft)$ we have $\eta_u(p):=T(p,k)$.
\item[(b)] If $w\in W$ then $\eta_u(w)$ is the natural action of $w$ on $\bbc[T]$.
\end{itemize}
\end{thm}
The next result is the version of Bernstein's isomorphism adopted to non-compact Riemannian
symmetric spaces:

\begin{thm}(\label{thm:etav}[\cite{C,He91a,He91,O4,OGHA}]
There is a {faithful representation $\eta_v=\eta_v(k):\bbh\to \textup{End}(C^\infty_c(A))$},
defined by the same map as above but with the operators acting on $C^\infty_c(A)$.
\end{thm}

\subsection{Comparison with the $p$--adic case; further developments}\label{sub:comp}
Let $X=G/K$ be a real Riemannian symmetric space. The "Archimedean Hecke algebra" $\ffR(\fg,K)$ (\cite{KV}) of the 
Riemannian symmetric pair $(G,K)$ is the algebra of (two-sided) $K$-finite distributions on $G$ with support on $K$. It is well--known that 
\begin{equation}
\ffR(\fg,K)\simeq U(\fg)\otimes_{U(\fk)} R(K)
\end{equation}
where $R(K)$ denotes the algebra of $K$-finite functions with support on $K$. Then $\ffR(\fg,K)$ is 
an associative algebra with respect to convolution (without unit, in general). 

This convolution algebra of distributions has a natural action (by convolution) on the space $C^\infty(G)_{K\times K}$ of 
two sided $K$-finite smooth functions on $G$.  

We remark that the indicator function $\gep_K\in R(K)$ of $K$ is an idempotent element. 
\begin{df}
The spherical algebra $\ffR(\fg/\!\!/K)$  is the algebra defined by $\ffR(\fg/\!\!/K)=\gep_K*\ffR(\fg,K)*\gep_K$. 
\end{df}
It is readily seen that $\ffR(\fg/\!\!/K)$ is an associative algebra 
with unit element $\gep_K$. Clearly $\ffR(\fg/\!\!/K)$ acts by convolution 
on the space of spherical functions $C(G/\!\!/K)$. 
It follows in the usual way that $\ffR(\fg/\!\!/K)$ is commutative, which also follows from the next Proposition:   
\begin{proposition}
Let $X=G/K$. 
The map $U(\fg)^K\to \ffR(\fg/\!\!/K)$ defined by $D\to \tilde{D}:=D* \gep_K$ induces an algebra isomorphism 
$\bbd(X)\simeq \ffR(\fg/\!\!/K)$. The space $C^\infty(X)^K$ corresponds to $C^\infty(G/\!\!/K)$ via lifting $f\to \tilde{f}$.
Then $\tilde{D}*\tilde{f}=(Df)\tilde{}$.  
\end{proposition}
\begin{proof}
Clearly $\ffR(\fg/\!\!/K)=\gep_K*U(\fg)\otimes_{U(\fk)}\bbc\gep_K=U(\fg)^K\otimes_{U(\fk)}\bbc\gep_K$.
Hence the algebra map  $U(\fg)^K\to \ffR(\fg/\!\!/K)$ given by $D\to D*\gep_K$ is surjective, and from $D*\gep_K=D\otimes_{U(\fk)}\gep_K$ we
see that the kernel equals $U(\fg)^K\cap U(\fg)\fk$, which yields the result. The action on smooth $K$-bi-invariant functions is immediate from the 
definitions.
\end{proof}
In the comparison with the Satake isomorphism in the $p$--adic case, it is more natural to define Harish--Chandra's 
homorphism not on $\bbd(X)$, but rather on the isomorphic algebra $\ffR(\fg/\!\!/K)$. We define: 
\begin{align}\label{eq:HCS}
\gc^\ffR_S: \ffR(\fg/\!\!/K)&\to S(\fa)^W\\
\nonumber D*\gep_K&\to\gc(D)
\end{align}
Recall Bernstein's extended isomorphism $\eta_B:\mathbf{H}(L,W,q)\to \mathbf{H}(G,B)$ in the $p$--adic case. 
We are not aware of the existence of any analog like this in the real case, hence we are settling  
for the present version, where we are only able to extend and deform the algebra of {\it radial parts} of operators in $\bbd(X)$.
Given the above, it is natural to transfer the "radial part algebra isomorphism" $\textup{Rad}$ on $\bbd(X)$ to a homomorphism 
$\textup{Rad}^\ffR$ on $\ffR(\fg/\!\!/K)$. In order to formulate this properly, we introduce the ring of operators on functions 
on $A$ (or $T$, 
or $T_u$, depending on the case) in which the Dunkl--Cherednik operators find their natural home. 

Let $T$ be the complex algebraic torus $T=A(\bbc)/(A(\bbc)\cap K_c)$ as above. The coordinate ring $\bbc[T]$ of $T$ is a subring $\bbc[T]\subset \bbc[A]$. 
Then we have $T=T_uA$, the polar decomposition, with $T_u=\Hom(P,S^1)$ the compact form, and $A:=\Hom(P,\bbr_+)\simeq \fa$, 
the real vector group of $T$. 

The Weyl denominator is a very important element of $\bbc[A]$ (or $\bbc[T]$), defined by 
and also on $A$ by: 
\begin{equation*}
\Delta(a):=\prod_{\alpha\in R^0_+}(a^{\frac12\alpha}-a^{-\frac12\alpha})\in \bbc[A]
\end{equation*}
where $R^0$ denotes the set of unmultipliable roots.  
\begin{df}
We denote by $\bba(T,W)\subset \textup{End}_{\bbc}(\bbc[T])$ the ring of algebraic regular differential--reflection operators 
on $T$, by which we mean the ring of differential operators on $T$ with coefficients in the ring $\bbc(T)[W]$  
(the twisted group ring of $W$, with coefficients in the field $\bbc(T)$ of rational functions on $T$) which preserve 
the ring $\bbc[T]$ of Laurent polynomials on $T$. We denote by $\bba(W\backslash T)$ the ring of polynomial partial differential 
operators on $W\backslash T\simeq \bbc^r$. Observe $\bba(W\backslash T)\e_+\subset \bba(T,W)$.

The restriction of elements of $\bbc[T]$ to functions on $T_u$ or on $A$ defines injective ring homomorphisms, whose images are denoted 
by $\bbc[T_u]$ and $\bbc[A]$ respectively. 
We sometimes denote by $\bba(A,W)$ or $\bba(T_u,W)$ the rings of restrictions of the operators from $\bba(T,W)$ to $A$ or $T_u$ respectively.   
\end{df}

It is not difficult to show that: 
\begin{proposition}(cf. \cite[Lemma 9.6]{OGHA})
The elements of $\bba(T,W)$ have coefficients in $\bbc[\Delta^{-1}][T][W]$. By restriction to $A$ and $T_u$ respectively, 
we have similar results for these cases. 
\end{proposition}
 
Observe that the Dunkl--Cherednik operators belong to this ring: $T(\xi,k)\in \bba(A,W)$. 
We define $\textup{Rad}^\ffR$ by:
\begin{align*}
\textup{Rad}^\ffR:\ffR(\fg/\!\!/K)&\to \bba(W\backslash A)\gep_+\subset \bba(T,W)\\
D*e_K&\to \textup{Rad}(D)\gep_+ 
\end{align*}


Let $\bbh_v^{D}(R_+,k)\subset \bba(A,W)$ denote the image of the injective homomorphism $\eta_v$ (as in Theorem \ref{thm:etav}), and let
$\bbh_v^{D,sph}(R_+,k)\subset \bbh_v^{D}(R_+,k)$ be its spherical algebra. We transfer the Harish--Chandra homomorphism $\gc_S^\ffR$ to 
$\bbh_v^{D,sph}(R_+,k)$ via $\textup{Rad}^\ffR$, to obtain an isomorphism 
\begin{equation}
\tilde{\gc}_S(k):\bbh_v^{D,sph}(R_+,k)\to S(\fa)^W
\end{equation}
characterised by $\tilde{\gc}_S(k)(D(p,k)\gep_+)=p$ for all $p\in S(\fa)^W$.

Now we are finally able to formulate the appropriate inversion of the Harish--Chandra isomorphism in the present context, in 
analogy with the explicit inversion of the Satake isomorphism Theorem \ref{thm:macsat} in the $p$--adic context:
\begin{cor}
The inversion of $\tilde{\gc}_S(k):\bbh_v^{D,sph}(R_+,k)\to S(\fa)^W$ is given by the isomorphism 
\begin{align*}
\epsilon(k):\bbc[\fa^*]^W=S(\fa)^W&\to \bbh_v^{D,sph}(R_+,k)\\
p&\to \eta_v(k)(p)e_+
\end{align*} 
(and similar for $\eta_u$).  
 \end{cor}
Recall that $\eta_v(k)$ (or $\eta_u(k)$) extends to an injective algebra homomorphism on the full graded affine Hecke algebra $\bbh(R_+,k)$, 
as we saw in the $p$--adic case with Bernstein's homomorphism. In other words, $\eta_v(k)$ and $\eta_u(k)$ are faithful representations of $\bbh(R_+.k)$.
A remarkable feature of the construction given here, is that the 
parameters $k_\ga$ of the restricted roots $\ga\in R_+$ need no longer be bound to be half the dimensions of the restricted root spaces in $\fg$. 
The representations $\eta_v(k)$ and $\eta_u(k)$ are merely specialisations of faithful representations $\eta_v$ and $\eta_u$ of the 
generic graded affine Hecke algebra $\bbh$.   
In particular, we may specialise these parameters to arbitrary complex values (constant on $W$-orbits in $R_+$).  
\begin{remark}
As in the $p$--adic case, two generalizations of the theory of zonal spherical functions for $(G,K)$ present 
themselves. First of all, the parameters $q_\alpha$ (or $k_\alpha$ in the real case) can now vary freely in $\bbc$. Secondly we can 
study the spectral problems for the commuting Dunkl--Cherednik operators $\{T(\xi,k)\}$ on $T_u$ or $A$, rather than for the 
$W$--invariant differential operators $\{D(p,k)\}$.  See \cite{HO}, \cite{He91}, \cite{OGHA}. 
\end{remark}
\begin{remark}
Studying such "generalised hypergeometric equations" has led to applications in mathematical physics (in conformal field theory \cite{C}, 
integrable models such as the Calogero-Moser system \cite{HO}, \cite{O4}, \cite{Ruis} and the Yang system \cite{HOH0}), 
algebraic geometry (study of certain period maps of moduli spaces \cite{HL}, \cite{CHL}) and algebraic combinatorics 
(the Macdonald conjectures \cite{O89}, \cite{He91}, \cite{C95}). 
\end{remark}
\begin{remark}
Oda \cite{Oda2} has improved much on the above remarks regarding a correspondence between spherical 
$\bbh(R_+,k)$-modules and spherical representations of $G$. Oda defined various functorial maps from the category of 
$\bbh(R_+,k)$-modules to the category of $(\fg,K)$-modules extending this. He achieved this by extending  
the radial part construction to non-invariant operators and functions, for certain $K$-types calles \emph{single petaled}. 
He also set up a detailed comparison between Helgason's Fourier decomposition of $C_c^\infty(X)$ and the 
decomposition of $C^\infty_c(A)$ studied in \cite{OGHA} (see also Section \ref{Section:HA}), allowing him to establish 
the mentioned functorial correspondences. 

For some classical groups, intricate exact functors from the category of admissible representations of $G$ 
to the category of modules over a graded affine Hecke algebra $\bbh(R_+,k)$ have been constructed, providing 
further analogy with the $p$--adic case. See \cite{Oda}, \cite{CT}. These functors map an irreducible unitary representation 
of $G$ to a $*$--unitary $\bbh(R_+,k)$-module which is irreducible or $0$.
\end{remark}
\begin{remark}
As was remarked at the end of Section \ref{sec:SI}, for $p$--adic reductive groups, the Borel-Casselman functor $V\to V^B$ 
from Iwahori--spherical smooth representations to $\bH(G,B)$-modules is an equivalence. Moreover, 
irreducible unitary representations are mapped to irreducible $*$--unitary modules of $\bH(G,B)$, where $*$ is the anti-linear 
anti--involution of $\bH(G,B)$ defined by $T_w^*=T_{w^{-1}}$. 
Remarkably, in this situation the converse also holds true \cite{BaMo}. 
In particular, the irreducible $K$--spherical unitary representations of $G$ are in bijection with the irreducible $*$--unitary $e_K$--spherical 
representations of $\bH(G,B)$. 

In the present situation for real reductive groups, one only has the partial results on functorial maps from representations of 
$G$ to $\bbh(R_+,k)$-modules mentioned above. 
In general we only have a much weaker correspondence between unitary representations of $G$ and $\bbh(R_+,k)$-modules: 
It is well--known that the irreducible unitary $K$--spherical representations $(V,\pi)$ of a noncompact real reductive group $G$, 
are in one to one correspondence with the elementary positive definite $K$--spherical functions $\varphi_\pi$ on $X=G/K$. 
Given such an elementary spherical function $\varphi_\pi$, we obtain a character of the spherical Hecke algebra $\ffR(\fg/\!\!/K)$, 
thus also of $\textup{Rad}^\ffR(\ffR(\fg/\!\!/K))$, the spherical algebra of $\bbh^D(R_+,k)$. 
Hence we obtain an irreducible spherical representation $(V^\bbh,\pi^\bbh)$ of $\bbh^D(R_+,k)$ corresponding to $(V,\pi)$.

By analogy with the situation for $p$--adic reductive groups, it would seem natural to expect that $(V^\bbh,\pi^\bbh)$ is unitary 
for $(\bbh^D(R_+,k),*)$. Indeed, in the special cases mentioned above, this follows from the properties of the constructed functorial 
maps, see \cite{CT}.  However, in these cases the converse does not always hold. Worse, for certain complex simple groups it 
is not always true that $(V^\bbh,\pi^\bbh)$ is $*$--unitary
\footnote{Dan Ciubotaru, private communication.}.  
\end{remark}
\begin{remark}
Another natural and compelling generalisation of the commuting differential--reflection operators $\{T(\xi,k)\}$ was provided
by Cherednik (see e.g. \cite{C95}), where they are further deformed to commuting {\it difference} reflection operators. This
development was originally motivated by 
their crucial role for the solution of the Macdonald conjectures, which can be understood as the explicit determination of the 
harmonic analysis relative to the Macdonald and Koornwinder families of orthogonal polynomials \cite{Mac96}, \cite{Koor92}. 
The Macdonald polynomials can be interpreted for certain values of their 
parameters as spherical functions on compact quantum symmetric spaces \cite{Nou}, \cite{Letz04}. 
The noncompact harmonic analysis associated to such systems of eigenfunction equations (related to the spectrum of noncompact quantum symmetric spaces) 
is an active area of research, with deep ramifications to representation theory and mathematical physics \cite{GKK}, \cite{KS}.
\end{remark}

\section{Generalised spherical transforms}

In this section we will give a brief overview of the joint eigenfunctions of the operators $T(\xi,k)$ and
the related spectral decompositions.

\subsection{Harmonic analysis; the Weyl density functions}

We will consider generalisations of the spherical Fourier transforms on Riemannian symmetric spaces.
The first generalisation is that we no longer assume that the root multiplicity parameters arise from the
dimensions of root spaces of restricted roots of a Riemannian symmetric pair. Secondly, we will not just
work with $|W|$--invariant functions on $A$ (the restrictions to $A$ of spherical functions on $X$) but
rather with general functions on $A$.

Let $\mathcal{K}$ be the space of multiplicity parameters \eqref{eq:cK}, and set
\begin{align*}
\Ck^{T_u}_+&=\left\{k\in\Ck\mid k_\ga+k_{\ga/2}\geq0, k_\ga\geq0,\, \alpha \in R^0\right\}\\
\Ck^A_+&=\left\{k\in\Ck\mid k_\ga+k_{\ga/2}\geq0,\, \alpha \in R^0\right\}
\end{align*}
Recall that we write $T=T_uA$ for the polar decomposition of $T$, where $T_u,A$ are given by \eqref
{eq:polar T}. If $k\in\Ck^{T_u}_+$ (resp. $k\in\Ck^A_
+$), define a $W$--invariant positive "Weyl density function" on $T_u$ (resp. $A$) by
\begin{align*}
\delta(k;t)&:=\prod_{\alpha>0}|t^{\frac12\alpha}-t^{-\frac12\alpha}|^{2k_{\alpha}}=
\prod_{\alpha>0}(2-t^{\alpha}-t^{-\alpha})^{k_{\alpha}}\\
\intertext{and}
\delta(k;a)&:=\prod_{\alpha>0}|a^{\frac12\alpha}-a^{-\frac12\alpha}|^{2k_{\alpha}}=
\prod_{\alpha>0}(a^{\alpha}+a^{-\alpha}-2)^{k_{\alpha}}
\end{align*}

\subsection{The inner product spaces}\label{ss:inner product}

We assume that $k\in \Ck^{T_u}_+$, so that $\delta(k;t)$ is an integrable function on $T_u$. 
Define a pre-Hilbert space structure $\langle\cdot,\cdot\rangle_k$ on $\bbc[T_u]$ by
\begin{equation*} 
\langle f,g \rangle_k=|W|^{-1}\int_{T_u} f(t)\overline{g(t)}\delta(k;t)dt=:|W|^{-1}\int_{T_u} f(t)\overline{g(t)}d\mu^T(k;t)
\end{equation*}
with $dt$ the normalized Haar measure on $T_u$.

Similarly, if $k\in \Ck^A_+$ then $\delta(k;a)$ is locally integrable function on $A$, and we define $(\cdot,\cdot)_k$ on $\Cc^\infty_c(A)$ by: 
\begin{equation*} 
( f,g )_k=|W|^{-1}\int_{A} f(a)\overline{g(a)}\delta(k;a)da
=:|W|^{-1}\int_{A} f(a)\overline{g(a)}d\mu^A(k;a)
\end{equation*} 
with $da$ the Haar measure on the vector group $A$ such that $2\pi Q^\vee$ has covolume $1$.
\subsection{$*$--structures and $*$--representations}
\begin{thm}\label{thm:adjoint}
The following adjunction formulas hold for the Dunkl--Cherednik operators $\{T(\xi,k)\}_{\xi\in\ft}$.
\begin{itemize}
\item[(i)] If $f,g\in \bbc[T]$, then \[\langle T(\xi,k)f,g \rangle_k=\langle f,T(\overline{\xi},k)g \rangle_k\]
\item[(ii)] If $f,g\in C^\infty(A)$ and one of $f,g$ belongs to $C_c^\infty(A)$, then
\[(T(\xi,k)f,g)_k=(f,{}^{w_0}T(-\overline{w_0(\xi)},k)g)_k\]
\end{itemize}
\end{thm}
The above formulas give rise to two distinct anti-linear anti--involutions on the graded affine Hecke algebra
$\bbh=\bbh(R_+,k)$.

\begin{thm}
The assignments
\[\xi^\bullet=\overline{\xi}\aand w^\bullet=w^{-1}\]
define an anti--linear anti--involution $\bullet$ of $\bbh$, thus turning it into a star algebra. Moreover,
the representation $\eta_u$ is pre--unitary with respect to $\bullet$ \ie for all $h\in \bbh$ and $f,g\in
\bbc[T]$  we have $\langle\eta_u(h)(f),g\rangle_k=\langle f,\eta_u(h^\bullet)g\rangle_k$.  
\end{thm}

\begin{thm}\label{thm:star}
The assignments
\[\xi^*=-w_0.w_0(\overline{\xi}).w_0\aand w^*=w^{-1}\]
define an anti--linear anti--involution $*$ of $\bbh$. The representation $\eta_v$ is pre--unitary with 
respect to $*$ \ie for all $h\in \bbh$ and $f,g\in C_c^\infty(A)$  we have $(\eta_v(h)(f),g)_k=(f,\eta_v
(h^*)g)_k$.  
\end{thm}

The unitary representation $\eta_u(k)$ on $L^2(T_u,\delta(t,k)dt)$ of $(\bbh(R_+,k),\bullet)$ can be decomposed
as a direct sum of spherical irreducible unitary $\bbh$--representations for $\bullet$. Similarly, the unitary
representation $\eta_v(k)$ on $L^2(A,\delta(k;a)da)$ of $(\bbh,*)$ can be decomposed as a direct Hilbert
integral of spherical irreducible unitary $\bbh$--representations with respect to  $*$. To study these
decompositions in more detail, we will first consider the joint eigenfunctions of the Dunkl operators.

\subsection{Harmonic analysis; joint eigenfunctions}\label{Section:HA}

The anti--involution $*$ of $\bbh$ is well--known from the theory of unitary representations of $p$--adic reductive
groups. It was shown by Barbasch and Ciubotaru \cite{BaCi} that these two anti--involutions are the only ones
which are equal to the inverse on $W$, provided $R$ is irreducible. The relevant $\bbh$-modules for the spectral
decomposition of the representations $\eta_u$, $\eta_v$ must in any case be unitary (for $\bullet$ in $\bbc[T]$,
for $*$ in $C_c^\infty(A)$), and $W$--spherical (the latter follows easily from \cite[Lemma 3.14]{OGHA}).
Additional requirements apply, as we will see.   

We remark that the classification of irreducible spherical unitary representations of $(\bbh,\bullet)$ is easy \cite{BaCi}.
The classification of the spherical $(\bbh,*)$--unitary modules is a difficult problem (a lot of interesting results in this direction 
have been obtained by Ciubotaru, cf. \cite{CiuUn}). 
For $k\in \Ck^A_+$,
the spectral measure of $L^2(A,\gd(a,k)da)$ has support in the unitary principal series \cite{OGHA}. 
For $k\in\Ck$ negative, with $|k|$ small, discrete series occurs (supported in the set of tempered spherical irreducible representations
of $(\bbh,*)$) \cite{ODS}. 

In order to study the decomposition of $L^2(A,\delta(k;a)da)$ in irreducible representations of $\bbh(R_+,k)$,
we need to study the joint eigenfunctions of the Dunkl--Cherednik operators $T(\xi,k)$ in $C^\infty(A)$.
\begin{thm}[\cite{OGHA}] Let $\Ck^+_\bbc=\{k\in\Ck\mid \textup{Re}(k)\in \Ck^A_+$\}.  
\begin{enumerate}
\item[(1)] There is an open $W$--invariant neighborhood $U\subset T_u$ of $1$, and a unique holomorphic
function $G$ on $\ft^*\times \Ck^+_\bbc\times A\cdot U$ 
with the following properties: 
\begin{itemize}
\item[(i)] $G(\gl,k;1)=1$.
\item[(ii)] $T(\xi,k)G(\gl,k;-)=\gl(\xi)G(\gl,k;-)$ for any $\xi\in\ft$.
\end{itemize}
\item[(2)] If $F$ is the holomorphic function on $\ft^*\times\Ck_\bbc^+\times A\cdot U$ defined by
\[F(\gl,k;t)=\frac{1}{|W|}\sum_{w\in W}G(\gl,k;wt)\]
then $F(\gl,k;t)$ is the hypergeometric function for root systems introduced in \cite{HO} (cf. \cite{HecSch}, \cite{OLect}).
The latter can be uniquely characterised by the requirements (\cite{HO,OGau}):
\begin{itemize}
\item[(i)] $F(\gl,k;-)$ is $W$--invariant and holomorphic in a neighborhood of $A\subset T$, and  $F(\gl,k;1)=1$.  
\item[(ii)] $D(p,k)F(\gl,k;-)=p(\gl)F(\gl,k;-)$ for any $p\in S(\ft)^W$.
\end{itemize}
\end{enumerate}
\end{thm}
We remark that both $G$ and $F$ have meromorphic continuations to $\ft^*\times \Ck\times A\cdot U$ \cite{OGHA}.
If $(R,k)$ corresponds to the Riemannian symmetric space $X=G/K$, with $A\subset X$ a maximal flat subspace,
then $F(\gl,k;t)=\varphi_\gl(t^{1/2})$ (with $t\in A\cdot U$, and $\varphi_\gl$ the elementary spherical function).
If $R=\{\pm\alpha,\pm2\alpha\}$ is the root system of type $\mathsf{BC}_1$, $F(\lambda,k_1,k_2;t)$ is Gauss' 
hypergeometric function ${}_2F_1(a,b,c;z)$ with $z=1/4-1/2(\alpha(t)+\alpha(t^{-1}))$, and 
\[
a=\frac{1}{2}(\gl+\rho(k))(\alpha^\vee)\qquad
b=\frac{1}{2}(-\gl+\rho(k))(\alpha^\vee)\qquad
c=\frac{1}{2}+k_\alpha+k_{2\alpha}
\]
which explains the name of the function $F$.

\subsection{Variations on Harish--Chandra's $c$--function}

The asymptotic expansion on $A_+$ (cf. (\ref{eq:Aplus})) of the hypergeometric function
$F$ was studied in \cite{HO} and \cite{OGau}, and this leads to the following  formula if $\textup{Re}(\gl)\in\fa_+$
and $a\in A_+$: 
\begin{equation}
c(\lambda,k)=\lim_{t\to \infty}a^{t(\rho(k)-\gl)}F(\gl,k;a^t)=\frac{\tilde c(\gl,k)}{\tilde c(\rho(k),k)}
\end{equation}
with 
\begin{equation*}
\tilde c(\lambda,k)=\prod_{\alpha \in R_+}
\frac{\Gamma\left( \lambda(\alpha^\vee)+\frac{1}{2}k_{\alpha/ 2}\right)}
{\Gamma\left( \lambda(\alpha^\vee)+\frac{1}{2}k_{\alpha/2}+k_\alpha
\right)}
\end{equation*}
This is a meromorphic continuation over the parameter space $\Ck$ of the marvellous Harish--Chandra $c$--function
(cf. \eqref{eq:lim}), in its product form as provided by the work of Gindikin and Karpelevi\v{c} \cite{GiKaI,GiKaII}.
To make the harmonic analysis with respect to $\bbh$ precise, we will now introduce some variations on the
$c$--function.
\begin{df}\hfill 
\begin{itemize}
\item[(a)] For $w\in W$, define $\delta_w: R_+\to\{0,1\}$ by
\[ \delta_w(\alpha ) = \left\{\begin{array}{ll} 0 &\text{ if }w(\alpha )>0\\1&\text{ if }w(\alpha )<0\end{array}\right. \]
\item[(b)] Define
\begin{equation*}
c_w^*(\lambda,k)=\prod_{\alpha\in R_+}
\frac{\Gamma\left( -\lambda(\alpha^\vee)-\frac{1}{2}k_{\alpha/2}-k_\alpha +\delta_w(\alpha)\right)}
{\Gamma\left( -\lambda(\alpha^\vee)-\frac{1}{2}k_{\alpha/2}+\delta_w(\alpha)\right)}
\end{equation*}
\item[(c)] Define
\begin{equation*}
\tilde c_w(\lambda,k)=\prod_{\alpha \in R_+}
\frac{\Gamma\left( \lambda(\alpha^\vee)+\frac{1}{2}k_{\alpha/ 2}+\delta_w(\alpha )\right)}
{\Gamma\left( \lambda(\alpha^\vee)+\frac{1}{2}k_{\alpha/2}+k_\alpha +\delta_w(\alpha )
\right)}
\end{equation*}
\item[(d)] Put $\tilde c(\lambda,k):=\tilde c_e(\lambda,k),\; c^*(\lambda,k):=c_{w_0}^*(\lambda,k)$.
\end{itemize}
\end{df}
\subsection{Compact case; decomposition of $L^2(T_u,\delta(k;t)dt)$}
The main reference for the material in this section is \cite{OGHA}. We assume that $k\in\Ck^{T_u}_+$.
Let us first look at the spectral decomposition in the compact case, i.e. the representation of $(\bbh^D_u(R_+,k),
\bullet)$
on $L^2(T_u,\delta(k;t)dt$. 
The first step is to consider the joint eigenfunctions of the Dunkl operators on $\bbc[T_u]$.

To this end, we introduce the following partial order $\lhd$ on the weight lattice due to
Heckman \cite{HeB}. For $\mu\in P$, let $\{\mu_+\}= P_+\cap W\mu$ and let $w^\mu\in W$ be the
element of maximal length such that $\mu=w^\mu\mu_+$. Given $\mu,\nu\in P$, we set
$\nu\lhd\mu$ if $\nu_+\leq \mu_+$, or $\nu_+=\mu_+$ and $w^\nu\leq_W w^\mu$, where
$\leq$ is the standard lexicograhic order on $P$ and $\leq_W$ the Bruhat order on $W$.

\begin{thm}(Heckman \cite{HeB}, see \cite[Sect. 2]{OGHA})
For any $\mu\in P$, there is a unique joint eigenfunction for the operators $\{T(\xi,k)\}_{\xi
\in\ft}$ of the form 
\[E(\mu,k;t)=
\sum_{\nu\lhd\mu}c_\nu t^\nu
\qquad\text{with $c_\mu=1$}\]
The corresponding joint eigenvalue is the element of $P\subset\ft^*$ given by
\Omit{
For $\mu\in P$, let $\{\mu_+\}= P_+\cap W\mu$ and let $w^\mu\in W$ be the element 
of maximal length such that $\mu=w^\mu\mu_+$. There exists a unique joint $T(\xi,k)$--eigenfunction 
of the form $E(\mu,k;t)=t^\mu+\sum_{\nu\lhd\mu}c_\nu t^\nu$ where $\nu\lhd\mu$ iff 
$\nu_+\leq \mu_+$, or $\nu_+=\mu_+$ and $w^\nu\leq_W w^\mu$. The joint eigenvalue is:
\begin{align*}
T(\xi,k)E(\mu,k)&=w^\mu(\mu_++\rho(k))(\xi)E(\mu,k)\\
&=(\mu+w^\mu(\rho(k)))(\xi)E(\mu,k)
\end{align*} 
}
\[w^\mu(\mu_++\rho(k))=\mu+w^\mu\,\rho(k)\]
\end{thm} 
From our condition $k\in\Ck_+^{T_u}$ it follows that $\rho(k)\in \fa^*_+$, which implies that the joint
eigenvalues $\{w^\mu(\mu_++\rho(k))\}_{\mu\in P}$ are all distinct.
Of course these polynomial eigenfunctions are - up to normalisation - just specialisations of the family $\{G(\gl,k;\cdot)\}$ of general eigenfunctions. 
More precisely, we have:
\begin{proposition} [\cite{OGHA}] If $k\in\mathcal{K}_+$, up to normalization the only joint $T(\xi,k)$--eigenfunctions in $\bbc[P]$ are 
of the form $G(w(\gl+\rho(k)),k)$ with $\gl\in P_+$ and $w$ maximal in the coset $wW_\gl$.
We have 
\begin{equation*}
E(w\gl,k)=\frac{\tilde{c}_{w_0}(\rho(k),k)}{\tilde{c}_w(w(\gl+\rho(k))}G(w(\gl+\rho(k)),k).
\end{equation*}
\end{proposition}
These polynomial eigenfunctions from by construction a complete orthogonal set in $L^2(T_u,\delta(k;t)dt)$. 
As the joint eigenvalues $\mu+w^\mu(\rho(k))$ are all real and - provided $k\in \Ck_+$ - distinct, these polynomials 
are mutually orthogonal.  

Given $\gl\in P_+$, let $w_\gl\in W$ be the longest element in the isotropy group $W_\gl$ of $\gl$, and define $\tilde{\gl}=\gl+w_\gl(\rho(k))$.
Let $\bbh_\gl=\bbh(R_{\gl,+},k)\subset \bbh(R_+,k)$ be the "Levi subalgebra" associated to $R_\gl$, the root system of the 
standard Levi Weyl group $W_\gl$. Let $\bbc_{\tilde{\gl}}$ denote the one dimensional $W_\gl$--spherical representation of $\bbh_\gl$, 
with joint eigenvalue $\tilde{\gl}$ for the commuting elements $\xi\in S(\fa)\subset \bbh_\gl$.  We now define:
\begin{equation}\label{eq:ind}
V_{\tilde{\gl}}:=\textup{Ind}_{\bbh_\gl}^\bbh(\bbc_{\tilde{\gl}})=\bbh\otimes_{\bbh_\gl}\bbc_{\tilde{\gl}}
\end{equation}
\begin{thm}[\cite{OGHA}] Let $k\in \Ck_+$.
The $\bbh$-module $V_{\tilde{\gl}}$ with $\gl\in P_+$ is irreducible and $W$--spherical. 
Moreover, there exists a unique $\bbh$-module isomorphism 
\begin{equation*}
j:V_{\tilde{\gl}} \xrightarrow{\simeq} \Ce_\lambda:=\langle E(w\gl,k)\rangle_{w\in W}\subset \bbc[T_u]
\end{equation*} 
such that $j(1\otimes 1)=E(\gl,k)$.

There exists a multiplicity free orthogonal decomposition:
\begin{equation}
L^2(T_u,\delta(k;t)dt)= \widehat\bigoplus_{\gl\in P_+}\Ce_\gl 
\end{equation}
of the unitary $(\bbh,\bullet)$-representation $L^2(T_u,\delta(k;t)dt)$ into 
spherical irreducible unitary subrepresentations (where the algebraic direct sum 
equals $\bbc[T_u]\subset L^2(T_u,\delta(k;t)dt)$).
The functions $E(\gl,k;\cdot)$ with $\gl\in P$ are mutually orthogonal, and: 
\begin{equation*}
\Vert E(w\gl,k)\Vert_k^2=
\frac{c_{ww_\gl}^*\left( -(\gl +\rho),k\right)}{\tilde c_{ww_\gl}(\gl +\rho ,k)}
\end{equation*}
\end{thm}

\subsection{Non-compact case; decomposition of $L^2(A,\delta(k;a)da)$}

The main reference for this section is \cite{OGHA}. We assume that $k\in \Ck_+^A$.  Let
$\bba=S(\fa)\subset\bbh$, a commutative subalgebra. For any $\gl\in\fa^*_\bbc$, consider
the one--dimensional $\bba$-module $\bbc_\gl$, and define the $\bbh$--module
\begin{equation*}
I_\gl:=\textup{Ind}_\bba^\bbh(\bbc_\gl) =\bbh\otimes_\bba\bbc_\gl
\end{equation*}
Then, there is a $*$--invariant nondegenerate sesquilinear pairing 
\begin{equation*}I_\gl\times I_{-\overline{\gl}}\to\bbc
 \end{equation*}
 defined by 
 $(w_1\otimes 1,w_2\otimes 1)=\delta_{w_1,w_2}$ for all $w_1,w_2\in W$. In particular, for $\gl\in i\fa^*$,
 $I_\gl$ admits a $*$ unitary structure defined by this pairing.  

We have the representation $\eta_v$ of $\bbh$ defined on the dense subspace $C^\infty_c(A)\subset L^2(A,\delta(k;a))$.  
This inclusion provides the representation space $C^\infty_c(A)$ with a pre-Hilbertian structure with respect to which 
$\eta_v$ is a $*$--representation, thanks to Theorem \ref{thm:star}. 
We have \cite[Thm. 9.10]{OGHA}:

\begin{thm} Let $k\in\Ck^A_+$. For any $\gl\in \fa^*_\bbc$, there is a morphism $\Cj_{\gl,k}:
I_\gl\to C^\infty(A)$ of $\bbh$--modules defined by $\Cj_{\gl,k}(w\otimes 1)={}^wG(\gl,k;-)$. 
\end{thm}
Let $\textup{PW}(\fa^*_\bbc)$ denote the space of Paley--Wiener functions on $\fa^*_\bbc$, which is the space of entire functions $\phi$ on 
$\fa^*_\bbc$ such that for every $N\in\bbn$ and $R\in\bbr_+$ there exist a constant $C_N\in\bbr_+$ such that 
\begin{equation}
|\phi(\gl)|\leq C_N(1+|\gl|)^{-N}e^{|\textup{Re}(\gl)|}
\end{equation}
We define some additional structure which uses the $W$--action on $A$. 
Given $a\in A$, denote by $C_a\subset A$ the convex hull of $Wa$ (i.e. the image under the exponential map 
of the convex hull of $W\log(a)\in\fa$). 
Let $H_a:\fa^*\to\bbr_{\geq 0}$ be the function $H_a(\gl):=\textup{max}_{y\in C_a}\gl(\log(y))$. 
We say that $\phi\in \textup{PW}_a(\fa^*_\bbc)$ if for every $N\in\bbn$ there exist a constant $C_N\in\bbr_+$ such that for all $\gl\in\fa^*_\bbc$:
\begin{equation}
|\phi(\gl)|\leq C_N(1+|\gl|)^{-N}e^{H_a(\gl)}
\end{equation}
Then it is not difficult to see that $\textup{PW}(\fa^*_\bbc)=\bigcup_{a\in A}\textup{PW}_a(\fa^*_\bbc)$. 

We now formulate a Paley--Wiener theorem in this context, based on the embedding maps $\Cj_{\gl,k}$ for $\gl\in\fa^*_\bbc$. 
First we define the Fourier transform $\Cf_k(f)$ of a compactly supported smooth function $f\in C_c^\infty(A)$ 
as follows.\footnote{$\Cf_k$ corresponds to the operator $\pi^{-1}\Cf$ in \cite[Section 8]{OGHA}.}
Let $\gl\in\fa^*_\bbc$. Then:   
\begin{equation}
\Cf_k(f)(\gl):=\int_{a\in A}f(a)G(-w_0\gl,k;w_0a)d\mu^A(k;a)
\end{equation}

The wave packet operator $\Cj_k$ is defined 
as follows\footnote{$\Cj_k$ corresponds to $\Cj_k\pi$ in \cite[Section 8]{OGHA}.} 
on the space of Paley--Wiener function on $\fa^*_\bbc$. 
Let $\phi\in \textup{PW}(\fa^*_\bbc)$. Then we define for $a\in A$: 
\begin{equation}
\Cj_k(\phi)(a):=\int_{i\fa^*}\phi(\gl)G(\gl,k;a)\prod_{\ga\in R^0_+}\left(1-\frac{k_\ga+\frac{1}{2}k_{\ga/2}}{\gl(\ga^\vee)}\right)d\nu(\gl)
\end{equation}
where 
\begin{equation}
d\nu(\gl)=\frac{(2\pi )^{-n}\tilde{c}_{w_0}^2(\rho(k),k)}
{\tilde{c}(\gl,k)\tilde{c}(w_0\gl,k)}d\textup{Im}(\gl)
\end{equation}
The main Theorem in this section is the refined Paley--Wiener Theorem in this context. For all $a\in A$ we have:
\begin{thm}(\cite[Thm. 9.13]{OGHA})
We have 
\[\Cf_k:C_c^\infty(A)\to \textup{PW}(\fa^*_\bbc)\aand\Cj_k:\textup{PW}(\fa^*_\bbc)\to C_c^\infty(A)\]
These operators are morphisms of $\bbh$-modules which are each other's inverse. Moreover, for each $a\in A$ 
we have 
\[\Cf_k(C_c^\infty(C_a))=\textup{PW}_a(\fa^*_\bbc)\aand\Cj_k(\textup{PW}_a(\fa^*_\bbc))=C_c^\infty(A)\]
\end{thm}
\begin{remark}(Comparison of $\Cf_k$ to the spherical transform $\Cf^G_{sph}$ of Harish--Chandra.)  
The restriction $\Cf_k^W$ of $\Cf_k$ to $C_c^\infty(A)^W$ yields an isomorphism $\Cf_k^W:C_c^\infty(A)^W\to \textup{PW}(\fa^*_\bbc)^W$. 

In case the parameters $k_\ga$ correspond to a Riemannian symmetric space $X$, the restriction map 
$\textup{Res}:C_c^\infty(G/\!\!/K)\to C_c^\infty(A)^W$ given by $\textup{Res}(f)=f|_A$ is bijective (this follows easily from the Paley--Wiener 
Theorem of the spherical transform $\Cf_{sph}^G$ of $X$), and we have for $f\in C_c^\infty(G/\!\!/K)$ that: 
\begin{equation*}
\Cf_k^W(\textup{Res}(f))=\Cf_{sph}^G(f). 
\end{equation*}
(See the discussion following Propositions \ref{prop:ImAb} and \ref{prop:transmut} and Corollary
\ref{cor:ATCcandS}.)

Observe that (for all $D\in \ffR(\fg/\!\!/K)$, $f\in C_c(G/\!\!/K)$, $\gl\in\fa^*_\bbc$):
\begin{equation}\label{eq:mix}
\Cf_{sph}^G(D*f)(\gl)=\tilde{\gc}_S(D)(\gl)\Cf_{sph}^G(f)(\gl), 
\end{equation}
where $\tilde{\gc}_S$ is the Harish--Chandra homomorphism from (\ref{eq:HCS}).
\end{remark}
It is not difficult to see from Theorem \ref{thm:adjoint}(ii) that $\Cf_k$ and $\Cj_k$ are also formally adjoint 
to each other. It follows that: 
\begin{thm}(\cite[Thm. 9.13]{OGHA})
The transform $\Cf_k$ has a unique extension to $L^2(A,d\mu^A)$, and this defines an isometric isomorphism 
of unitary $*$--representations of $\bbh(R_+,k)$:
\begin{equation}
L^2(A,d\mu^A)\simeq \int_{i\fa^*_+}^\oplus I_\gl\nu(\gl)
\end{equation}
\end{thm}
\subsection{Further applications and developments}
Let us make some comments on some of the many very interesting further developments from here. 
\begin{remark} The orthogonal polynomials $P(\gl,k;\cdot)$, and their nonsymmetric counterparts $E(\gl,k;\cdot)$ have been generalised 
by Macdonald \cite{Ma1}, \cite{Mac96} and Koornwinder \cite{Koor92}. This family is a common generalisation of the Askey--Wilson 
orthogonal polynomials in one variable, and the root system orthogonal polynomials $P(\gl,k;\cdot)$. They are joint eigenfunctions 
of a commutative ring of difference-reflection operators which arises from a representation of the so-called 
double affine Hecke algebra, constructed by I. Cherednik \cite{C}, \cite{C95}. These polynomials depend on a set of 
parameters $q_i$ and $t$, and in the limit for $t\to 1$ (with $q_i$ certain powers of $t$) we recover the $P(\gl,k;\cdot)$.

This creates a uniform framework for the Macdonald-Koornwinder polynomials and the double affine Hecke algebra, 
in which the $p$--adic case and the (compact) real case coexist. The double affine Hecke algebra is built from two ordinary affine 
Hecke algebras sharing a finite type Hecke algebra in common, and whose translation parts are in duality. The Cherednik Fourier 
transform gives an isomorphism of the double affine Hecke algebra for $G$ to the double affine Hecke algebra of its Langlands 
dual group $G^\vee$. Macdonald has made the remarkable observation that the spherical functions of a $p$--adic group and their 
real counterparts seem to exist on the opposite sides of this duality. An in depth analysis of this situation can be found in \cite{BKP}.

One of the many remarkable aspects of the double 
affine Hecke algebra is its large group of automorphisms, which cannot be seen at the level of the affine Hecke algebra. These 
have important consequences for the theory of the Macdonald-Koornwinder polynomials. Important contributions and insights 
were made by Sahi \cite{Sah}, Ion and Sahi \cite{IoSa}, and Sahi, Stokman, and Venkateswaran \cite{SSV}. 

Formidable work and a lot of activity by many mathematicians has been aimed at attempting to understand Cherednik's double 
affine Hecke algebra in terms of more conventional geometric, representation theoretic and mathematical physical structures. 
We do not attempt to give a comprehensive overview, as there are simply too many directions and contributions 
to be mentioned and this would lead us too far afield. The authors content themselves mentioning some of the authors involved in these 
endeavors (and apologises for the many omissions and oversights): Braverman, Etingof and Finkelberg \cite{BEF}, 
Jordan \cite{J}, Kato, Khoroshkin and Makedonskyi \cite{KKM}, Kapranov \cite{K}, Oblomkov \cite{ob}, Varagnolo and Vasserot \cite{VV}, \cite{VV2}, 
Wen \cite{W}, Yun \cite{Y}.
\end{remark}
\begin{remark}
 The commutativity of the spherical algebras $\mathbf{H}(G,K)$ and $\eta(e_+\bbh(R_+,k)e_+)=\eta(S(\fa)^W)
 =\langle D(p,k)\mid p\in S(\fa)^W\rangle$ can be interpreted 
as the complete integrability of two famous $1$-dimensional $n$-particle quantum systems: The Yang system and the 
Calogero-Moser system respectively. 
The "Macdonald generalisation" of these systems, known as the Macdonald-Ruijsenaars-Schneider integrable model \cite{Ruis} or 
relativistic Calogero-Moser system, is the subject of an active field of research. 
The open quantum Toda chain is known to be a confluent limit of such systems \cite{DE}. 
Further elliptic generalisations have been found and studied \cite{Dell}, \cite{KoHi}.
\end{remark}
\begin{remark}
The spectral decomposition of the (densely defined) representation of the spherical algebra $e_+.\bbh.e_+\subset \bbh$ on 
$L^2(A,\gd(k;a)da)^W$ 
for parameters in $\Ck^A_+$ has also been proved by P. Delorme \cite{Del}, by introducing an appropriate Schwartz space, more 
in the spirit of the original approach by Harish--Chandra. It would be interesting to generalise this approach to 
the representation $\eta_v$, densely defined on $L^2(A,\gd(k;a)da)$. 

The spectral decomposition of the representation of the spherical algebra $e_+.\bH.e_+$ on $L^2(A,\gd(k;a)da)^W$  
is also well understood for negative parameters $k$ such that $\gd(k;a)da$ is still locally integrable \cite{ODS}. 
The spectrum is now no longer purely continuous, but lower dimensional spectral series appear. 
It woud be interesting to extend the Paley--Wiener Theorem and the definition of the Schwartz space 
to this context.
\end{remark}
\begin{remark}
A deep aspect of the theory of these higher rank hypergeometric functions $F(\gl,k)$ is their role as period 
map for certain algebraic surfaces, for very special values of the parameters (analogous to 
the interpretation of e.g. ${}_2F_1(1/2,1/2,1;z)$ as a period integral of elliptic curves). 

Couwenberg, Heckman and Looijenga \cite{CHL} have made deep studies of the parameters where the monodromy group 
of (subsystems of) the hypergeometric system becomes an arithmetic group, and Heckman and Looijenga \cite{HL} 
found period map interpretations for certain $F(\lambda,k)$ for the root system of $R=E_8$. 
\end{remark}

\section{Shift operators}

\newcommand {\ep}{\epsilon_+}

We have studied the representations $\eta_v(k)$ and $\eta_u(k)$ of the graded affine Hecke algebra
$\bbh(R_+,k)$ on $C_c^\infty(A)$ and $\bbc[T_u]$ respectively. For special half--integral values of the
parameters $k=(k_\ga)$, their restrictions 
to the spherical algebra $e_+\cdot
\bbh(R_+,k)\cdot e_+\simeq S(\fa)^W$ can be thought of as taking the radial part on $A\subset X=G/K$ of the 
inverse of Harish--Chandra's isomorphism. In this present section, we study the relation between these
representations of $\bbh(R_+,k)$ for different values of $k$. 

\subsection{The Abel transform}

If $k\in \Ck$ arises from a Riemannian symmetric space $X=G/K$, the Abel transform $F_f$ provides
such a relation on the level of the spherical algebra, between $k$ and the parameter $k=0$. Recall the
restriction map $\textup{Res}: C_c^\infty(G/\!\!/K)\to C_c^\infty(A)^W$. It follows from the Paley--Wiener
theorem and Helgason's Paley--Wiener Theorem \cite{Helgason} that for any $f\in C_c^\infty(A)^W$,
there is a unique $\tilde{f}\in C^\infty_c(G/\!\!/K)$ such that $\textup{Res}(\tilde{f})=f$. Thus, $\textup
{Res}$ is bijective and we may define
\begin{align*}
\Ca_k:C_c^\infty(A)^W&\to C_c^\infty(A)^W\\
f&\to\Cf_{\tilde{f}}
\end{align*} 
Then, by Harish--Chandra's result (Prop. \ref{prop:ImAb}), the following ``transmutation property'' holds
for all $p\in S(\fa)^W$ (with $D(p,k):=\eta_v(k)(p)$): 
\begin{equation}
\Ca_k(\eta_v(k)(p))=\Ca_k(D(p,k)f)=\partial(p)\Ca_k(f)=\eta_v(0)(p)(\Ca_k(f))
\end{equation}

\subsection{Hypergeometric shift operators}
These are $W$--invariant \emph{differential} operators which have a similar
transmutation property, this time not between $k$ and $0$ but between $k$ and $k+b$, where $b$ is
a fixed translation vector in $\Ck$. The fact that we require the shift operator to be a differential one
implies that the translation $b\in\Ck$ needs to be an element in a certain integral lattice $\Ck_\bbz\subset
\Ck$. For simplicity of presentation we will concentrate on the main case, where $b=\pm1\in\Ck$, where
$1$ is given by the following
\begin{df}
We denote by $1\in\Ck$ the multiplicity parameter such that $1_\ga=1$ if $2\ga\not\in R$, and otherwise $1_\ga=0$. 
\end{df}
The defining property of a hypergeometric shift operator $G$ is that for all $k\in\Ck$ and $p\in S(\ft)^W$,
one should have
\begin{equation*}
G_\pm(k)\circ D(p,k)=D(p,k\pm 1)\circ G_\pm(k)
\end{equation*}
Hypergeometric shift operators $G_\pm(k)$ are effective tools in the study of the hypergeometric system. To list 
a few nontrivial results in which they play a crucial role (cf. \cite{O89}): 
\begin{itemize}
\item[(a)] The inner product $\langle1,1\rangle_k$ has been computed using these operators (leading to Macdonald's constant term conjecture, 
and the Mehta-Dyson conjecture).
\item[(b)] The norms and values at $1$ of the Jacobi polynomials $P(\gl,k)$ (Macdonald's norm and evaluation formulas). 
\item[(c)] The computation of the Bernstein-Sato polynomial of the discriminant $d=(\prod_{\ga\in R^0_+}\ga)^2\in \bbc[\ft]^W$ of $W$.
\item[(d)] The proof that $F(\gl,k;1)=1$, with the given normalisations \cite{OGau}. 
\item[(e)] The inversion of the Abel transform in case the multiplicities of the restricted roots are all even \cite{beer}.
\end{itemize}

\subsection{Heckman's construction} 

Heckman \cite{He91} introduced an elementary method to obtain the hypergeometric shift operators which
were originally constructed in \cite{O4} by transcendental methods. In \cite{OGHA}, a slight variation of
Heckman's method was discussed.

Let $k\in\Ck$ and let $\Delta=\prod_{\ga\in R^0_+}(e^{\alpha/2}-e^
{-\alpha/2})\in\bbc[T]$ be the Weyl denominator.
\begin{itemize}
\item[(a)] Set $\pi^\pm(k):=\prod_{\ga\in R^0_+}(\ga^\vee\pm (k_\ga+\frac{1}{2}k_{\ga/2}))\in S(\ft)$, and let 
$\epsilon^+$  (resp. $\epsilon^-$) be the central idempotent in $\bbc[W]$ corresponding to the trivial (resp. sign)
character. Then, the following holds {in $\bH(R_+,k)$}
\begin{equation}\label{eq:em ep}
\epsilon^\mp\cdot\pi^\pm(k)\cdot\epsilon^\pm=\pi^\pm(k).\epsilon^\pm
\end{equation}
\item[(b)] Set $G_+(k):=\Delta^{-1}\circ T(\pi^+(k),k)|_{\bbc[T]^W}$.
This is a $W$--invariant polynomial differential operator on $T$ such that, for any $p\in S(\ft)^W$ we have:  
$G_+(k)\circ D(p,k)=D(p,k+1)\circ G_+(k)$  {(Hence $G_+(k)$ is a raising operator)}.
\item[(c)] Define $G_-(k):=T(\pi^-(k-1),k-1)\circ \Delta|_{\bbc[T]^W}$.
This is a $W$--invariant polynomial differential operator on $T$ such that, for any
$p\in S(\ft)^W$ we have: $G_-(k)\circ D(p,k)=D(p,k-1)\circ G_-(k)$  {(Hence $G_-(k)$
is a lowering operator)}.
\end{itemize}

\subsection{Nonsymmetric shift operators}

A natural question in the present discussion concerns the existence of shift operators
for the Dunkl--Cherednik operators $T(\xi,k)$ themselves, not just for the $W$--invariant
differential operators $D(p,k)$. 

Recall that a {\it differential--reflection} operator on $T$ is a $\bbc$--linear endomorphism
of $\bbc(T)$ of the form $D=\sum_{w\in W}D_w w$, where each $D_w$ is a differential
operator on $T$ whose coefficients are rational functions. We announce the following
result.\footnote{The results in this section were announced at Patrick Delorme's 70th
birthday conference in Luminy in June 2023, and the Harish--Chandra
centenary conference held in Allahabad in October 2023 (a video recording of the latter talk is available at:
{\sf https://www.youtube.com/playlist?list=PLt4li39O1V1Vq1yzPdIsOp7yo92XOsHh0}).}

\begin{thm}[\cite{OT}]\label{th:OT}
There is a differential--reflection operator $S$ on $T$ with coefficients in $\bbq
[\Ck][T](\Delta^{-1})$, such that 
\begin{itemize}
\item[(i)] $S\in \End_{\bbq[\Ck]}(\bbq[\Ck][T])$, that is $S$ preserves polynomials.
\item[(ii)] $S(k)\circ T(\xi,k)=T(\xi,k+1)\circ S(k)$ for any $\xi\in \ft$. 
\item[(iii)] $S(k)|_{\bbc[T]^W}=G_+(k)\in \bA(W\backslash T)$. 
\end{itemize}
Moreover, $S$ is the unique $\bbq(\Ck)$--linear endomorphism of $\bbq(\Ck)[T]$
satisfying (ii) and (iii). 
\end{thm}

We call $S$ the {\it nonsymmetric shift operator}.

\Omit{
We call $S$ the {\it nonsymmetric shift operator}. We note in passing that a similar
nonsymmetric {lowering} operator (i.e. the nonsymmetric version of $G_-(k)$)
does \emph{not} exist, even as a linear operator preserving polynomials \cite{OT}.\\
}

\subsection{Adjoints of nonsymmetric shift operators}

Interestingly, a nonsymmetric lowering operator, that is a $\bbq(\Ck)$--linear
endomorphism of $\bbq(\Ck)[T]$ satisfying property (ii) with $k+1$ replaced
by $k-1$, and restricting to $G_-(k)$ on invariant polynomials does \emph{not}
exist \cite{OT}. However, an operator satisfying a variant of the restriction
property (iii) does exist, and may be obtained from $S(k)$ using adjoints
as follows.

Let $S(k)^\bullet$ be the formal adjoint of $S(k)$ \wrt the inner products $\langle
\cdot,\cdot\rangle_{k}$ and $\langle\cdot,\cdot\rangle_{k+1}$ on $\Cc^\infty(T_u)$
(see \ref{ss:inner product}), so that $\langle S(k)^\bullet f,g\rangle_{k+1}=\langle
f,S(k)g\rangle_{k}$ for all $f,g\in C^\infty(T_u)$. Then, it is easy to see that 
$S_-(k):= S(k-1)^{\bullet}$ is a differential--reflection operator which 
preserves polynomials, satisfies the transmutation property
\[S_-(k)\circ T(\xi,k)=T(\xi,k-1)\circ S_-(k)\]
together with the {\it left} restriction property $\ep S_-=\ep  G_-(k)$. This follows
because $T(\xi,k)^\bullet=T(\overline{\xi},k)$ by Theorem \ref{thm:adjoint}, and
$G_+(k-1)^\bullet=G_-(k)$.\Omitcomment{I slightly reformulated since it actually requires a little argument 
to show that the formal adjoint preserves polynomials. The reason is that by the 
triangularity relative to $\leq_W$, $S_-(k)(e^\lambda)$ is orthogonal to all $E(\mu,k)$  
for $|\mu|>>|\lambda|$.}

Similarly, let $S(k)^*$ be the formal adjoint of $S(k)$ \wrt the inner products $(\cdot,
\cdot)_{k}$ and $(\cdot,\cdot)_{k+1}$ on $C_c^\infty(A)$. Then one can prove that
the operator $\wt{S}_-(k):={}^{w_0}(S(k-1)^*)$ is a differential--reflection operator
which preserves  $C_c^\infty(A)$,\Omitcomment{ERIC: This requires the nonsymmetric
Paley--Wiener Theorem of \cite{OGHA}} satisfies the transmutation property
\[\wt{S}_-(k)\circ T(\xi,k)=T(\xi,k-1)\circ \wt{S}_-(k)\]
and the left restriction property
$\ep \wt{S}_-=(-1)^{|R^0_+|}\ep  G_-(k)$.\footnote{It is shown in \cite{OT} that
$\wt{S}_-(k)$ and $S_-(k)$ coincide up to a sign as differential--reflection operators.
Specifically, $S(k)^*$ extends to $C^\infty(A)$, and the restriction of ${}^{w_0}(S
(k)^*)$ to $\bbc[T]\subset C^\infty(A)$ is equal to $(-1)^{|R^0_+|}S(k)^\bullet$.}

\subsection{Compositional identities}

Motivated in part by the above, we consider the composition of $S(k)$ with its formal
adjoints.

\begin{thm}[\cite{OT}]
The following identities hold. 
\begin{enumerate}
\item On $\bbc[T]$ one has:
\begin{equation}
S(k)^\bullet\circ S(k)=\prod_{\ga\in R^0_+}(T_{\ga^\vee}(k)+k^0_\ga)\circ (T_{\ga^\vee}(k)-k^0_\ga-1)\\[1.2ex]
\end{equation}
\item On $C_c^\infty(A)$ one has:
\begin{equation}
{}^{w_0}(S(k)^*)\circ S(k)=\prod_{\ga\in R^0_+}(T_{\ga^\vee}(k)+k^0_\ga)\circ (-T_{\ga^\vee}(k)+k^0_\ga+1)
\end{equation}
\end{enumerate}
\Omit{
\begin{align*}\label{eq:comp}
S(k)^\bullet\circ S(k)&=\prod_{\ga\in R^0_+}(T_{\ga^\vee}(k)+k^0_\ga)\circ (T_{\ga^\vee}(k)-k^0_\ga-1)\\[1.2ex]
{}^{w_0}(S(k)^*)\circ S(k)&=\prod_{\ga\in R^0_+}(T_{\ga^\vee}(k)+k^0_\ga)\circ (-T_{\ga^\vee}(k)+k^0_\ga+1)
\end{align*} 
where, for any unmultipliable root $\ga\in R^0$, we set $k^0_\ga=k_\ga+\frac{1}{2}k_{\ga/2}$,
with $k_{\ga/2}=0$ if $\ga/2\notin R$.}
\end{thm}

The above identities extend those satisfied by the hypergeometric shift operator $G(k)
_+$. These can be easily obtained from \cite[Lemma 5.10, Def. 5.11, Thm. 5.13(b)]
{OLect}, and read as follows\Omitcomment{I wonder if the rhs of the $GG$ identities should
be written either with a restriction to invariant polynomials sign, or (right) sandwiched by $\ep$.
This seems a good idea for consistency, but also for deriving the $GG$ identities from
the $SS$ ones below.}
\begin{align*}
G(k)_+^\bullet\circ G(k)_+	&=\prod_{\ga\in R^0_+}(T_{\ga^\vee}(k)+k^0_\ga)\circ(T_{\ga^\vee}(k)-k^0_\ga)|_{\bbc[T]^W}\\[1.1ex]
G(k)_+^*\circ G(k)_+		&=\prod_{\ga\in R^0_+}(T_{\ga^\vee}(k)+k^0_\ga)\circ(-T_{\ga^\vee}(k)+k^0_\ga)|_{C_c^\infty(A)^W}
\end{align*}

Note the subtle difference between the two sets of identities. The right--hand side of
those involving $G(k)_+$ are $W$--invariant, as expected, while those involving $S(k)$
are not. Nevertheless, the latter give rise to the former upon being left and right multiplied
by the idempotent $\ep$. Indeed, the left and right restriction properties of $S(k)$ 
and its adjoints imply that
\[\begin{split}
G(k)_+^\bullet\circ G(k)_+
&=
\ep S(k)^\bullet\circ S(k)\ep\\
&=
\ep \prod_{\ga\in R^0_+}(T_{\ga^\vee}(k)+k^0_\ga)\circ (T_{\ga^\vee}(k)-k^0_\ga)\ep
+\ep T(p,k) \pi_k^+ \ep
\end{split}\]
where
\[ p = 
\sum_{m=0}^{|R_+^0|-1}
(-1)^{|R_+^0|-m}\sum_{\substack{\overline{R}\subset R_+^0:\\|\overline{R}|=m}}
\prod_{\alpha\in\overline{R}} T_{\alpha^\vee}(k) \]
This yields the required result since $\pi_k^+\ep=\epsilon_- \pi_k^+\ep$ by \eqref
{eq:em ep}, and $\ep T(p,k)\epsilon_{-}=0$ for all $p\in S(\mathfrak{a})$ with
$\textup{deg}(p)<|R_+^0|$.

\subsection{Rank 1}

\newcommand {\h}{\mathfrak h}
\newcommand {\IZ}{\mathbb{Z}}
\newcommand {\IC}{\mathbb{C}}
\newcommand {\sfBC}{\mathsf{BC}}
\newcommand {\half}[1]{\frac{#1}{2}}

As an illustration, let $R=\{\pm\alpha,\pm\alpha/2\}$ be the root system of type $\sfBC_1$.
The corresponding dual root system is $R^\vee=\{\pm\alpha^\vee,\pm2\alpha^\vee\}$, the
weight lattice is $P=\IZ\alpha/2$, and the torus $T=\Hom_{\IZ}(P,\IC^\times)$ has the
natural coordinate $X=e^{\alpha/2}$.

A direct computation shows that the nonsymmetric shift operator of Theorem \ref{th:OT} is
given by
\[\begin{split}
S 
&=
\Delta^{-1}\left(X\partial_{X}-\frac{1-s}{1-X^{-2}}\right)\\
&=
\Delta^{-1}\left(T(\alpha^\vee;k)+k_{\alpha}+\half{1}k_{\alpha/2}
-k_{\alpha/2}\frac{1-s}{1-X^{-1}}
-(2k_\alpha+1)\frac{1-s}{1-X^{-2}}\right)
\end{split}\]
where $\Delta=X-X^{-1}$.

Similarly, the hypergeometric shift operator for the basic shift $(k_{\alpha/2},k_\alpha)
\to(k_{\alpha/2}+2,k_\alpha-1)$ (as considered in \cite[p. 26]{O3}) admits a nonsymmetric
version. It is given by
\[S_{(2,-1)} =
\frac{X+1}{X-1}X\partial_{X}-\frac{X}{(X-1)^2}(1-s)+k_{\alpha}-\half{1}
\]
Contrary to the shift by $1\in\Ck$, the corresponding nonsymmetric shift operator with the opposite 
shift $(k_{\alpha/2},k_\alpha)
\to(k_{\alpha/2}-2,k_\alpha+1)$ also exists now. It is given by
\[S_{(-2,1)} =
\frac{X-1}{X+1}X\partial_{X}+\frac{X}{(X+1)^2}(1-s)+\left(k_{\alpha}+k_{\alpha/2}-\half{1}\right)
\]

\Omit{
\subsection{Rank 1}

As an illustration, assume that $R=\{\pm\alpha\}$ is of rank 1. Let $\alpha\in\h^*$ be a choice
of positive root, and $\alpha^\vee\in\h,\lambda=\alpha/2\in\h^*$ the corresponding coroot and
fundamental weight. Let $P=\IZ\lambda\subset\h^*$ be the weight lattice, and $T=\Hom_{\IZ}
(P,\IC^\times)$ the corresponding torus. We denote the characters of $H$ by $\{e^\mu\}_{\mu
\in P}$. Thus, $X=e^\lambda$ is a coordinate on $T$, and $e^\alpha=X^2$.

Then, a direct computation shows that the non--symmetric shift operator is given by
\[\begin{split}
S 
&=
\Delta^{-1}\left(T(\alpha^\vee,k)+k-(2k+1)\frac{1-s}{1-X^{-2}}\right)\\
&=
\Delta^{-1}\left(X\partial_{X}-\frac{1-s}{1-X^{-2}}\right)
\end{split}\]
}

\end{document}